\documentclass[a4paper,12pt]{article}
\usepackage{geometry}
\usepackage{amsmath}
 \usepackage{amssymb}
 \usepackage{amsthm}
 \usepackage{mathrsfs}
\newtheorem{definition}{Definition}[section]
\newtheorem{theorem}{Theorem}[section]
\newtheorem{proposition}{Proposition}[section]
\newtheorem{lemma}{Lemma}[section]

\newtheorem{remark}{Remark}[section]
\newtheorem{assumption}{Assumption}[section]
\newtheorem{corollary}{Corollary}[section]
\newtheorem{notation}{Notation}[section]

\newtheorem{example}{Example}

\begin{document}
\title{Stochastic Time-Periodic Tonelli Lagrangian \\on Compact Manifold }
\author{Liang Chen \thanks{Mathematics Department, University Tor Vergata of Roma, Roma, Italy}\thanks{Email Address: chen@axp.mat.uniroma2.it}}
\date{}
\maketitle

\begin{abstract}
	In this paper, we study a class of time-periodic stochastic Tonelli Lagrangians on compact manifolds. Precisely, we discuss the stochastic  Mane Critical Value, prove the existence of stochastic Weak KAM solutions of the related Hamilton-Jacobi equation. Furthermore, we survey the global minimizer.
\end{abstract}

\section{Introduction}

\quad\quad Time-periodic Tonelli Lagrangians have been extensively studied in recent years by John Mather\cite{mather1991action}, Ricardo Mane \cite{mane1996generic} \cite{mane1992minimizing}, Patrick Bernard \cite{bernard2008dynamics} \cite{bernard2002connecting}, Daniel Massart \cite{massart2006subsolutions}, etc. It is closely related to calculus of variations, the weak KAM theory \cite{fathi2008weak}, Aubry-Mather theory and Optimal Transport(\cite{villani2008optimal}). In this paper, we study the time-periodic Tonelli Lagrangian by adding a stochastic variable. We discuss the measurability of Mane-Critical Value, prove the existence of stochastic viscosity solution for a system of Hamilton-Jacobi equations, 
and describe the global minimizer.

Bacis Setting: Let $M$ be a compact, connected Riemannian manifold,  $TM$ its tangent bundle. $(\Omega,\mathscr{F},\mathbb{P})$ is a probability space, $\mathscr{F}$ is the  $\sigma-$algebra of $\Omega$. $\theta: R \times \Omega \to \Omega$ is a measurable function with skew-product structure. Let $\mathscr{B}$ be the Borel algebra of $\mathbb{R}$. Then the $\sigma-$algebra of $\mathbb{R}\times\Omega$ is given by $\mathscr{B}\otimes \mathscr{F}$. The stochastic time-periodic lagrangian $L:TM \times T \times \Omega \to R$,  is a measurable function.  For each $\omega \in \Omega$, $L(\cdot,\cdot,\cdot,\omega):TM \times T \to R$ is a Tonelli Lagrangian, and for each $(x,v,t)\in TM \times T$, $L(x,v,t,\cdot) : \Omega \to R$ is a measurable function. For each $(x,v,t,\omega) \in TM \times T \times \Omega$, $s \in R$, we assume
\[
L(x,v,t+s,\omega)=L(x,v,t,\theta(s,\omega))
\]

Let $-\infty < s < t < \infty$, $x,y \in M$, $\omega \in \Omega$. $AC(s,x;t,y)$ is the space of absolutely continuous curves in $R \times M$, connecting $(s,x)$ and $(t,y)$  The stochastic Lagrangian action $A^{\omega}(s,x;t,y)$ is
\[
A^{\omega}(s,x;t,y)=\inf_{\gamma \in AC(s,x;t,y)}\int^{t}_{s}L(\gamma(\sigma),\dot{\gamma}(\sigma),\sigma,\omega)d\sigma
\]

Here, we study the following topics:

1. Mane-Critical Value

2. The existence of stochastic viscosity solution for a stochastic system of Hamilton-Jacobi equations.

3. global minimizer.

For each $\omega \in \Omega$, $\Phi^{\omega}$ is the set of invariant measure of Hamiltonian flow in $TM \times T$ for the Tonelli Lagrangian $L(\cdot,\cdot,\cdot,\omega)$.  The Mane-Critical Value for $L(\cdot,\cdot,\cdot,\omega)$,  is 
\begin{equation}\nonumber
\alpha(\omega)=\inf_{\mu \in \Phi^{\omega}} \int_{TM \times T}L(x,v,t,\omega)  d \mu
\end{equation}

The first main result is: 
\begin{theorem}
	The function $\alpha(\omega)$ is a measurable function on $\Omega$. In particular, if $\{\theta(t): \Omega \to \Omega|t \in R \}$ is ergodic, $\alpha(\omega)$ is constant almost everywhere on $\Omega$.
\end{theorem}

Now,  we define the Lax-Oleinik Operator, for $u: M \to R$, $(x,t) \in M \times T$, $\lambda \in R$, $\omega \in \Omega$,
\[
T^{\omega}_{\lambda}u(x,t)=\min_{y \in M}\{u(y)+A^{\omega}(t-\lambda,y;t,x)+ \lambda \alpha(\omega)\}
\]

The Hamiltonian, $H: T^{*}M \times T \times \Omega \to R$, is defined as,
\[
H(x,p,t,\omega)=\sup_{v \in T_x M}\{p\cdot v - L(x,v,t,\omega)\}
\]

The second main result is a weak KAM-type theorem,
\begin{theorem}
	For each $u:M \to R$, $(x,t) \in M \times R$, define:
	\[
	u^{\omega}(x,t)=\liminf_{\lambda \to +\infty}T^{\omega}_{\lambda}(u)(x,t) 
	\]
	Then, the following holds true: 
	
	(i)For all  $\omega \in \Omega$, $u^{\omega}(x,t)$ is a viscosity solution of the Hamilton-Jacobi equation:
	\[
	\partial_t u(x,t) + H(x,\partial_x u(x,t),t,\omega)= \alpha(\omega)
	\]
	
	(ii)For all $(x,t) \in M \times R$, $u^{\omega}(x,t)$ are measurable functions on $\Omega$.
	
	(iii)For all $x \in M$, $s,t \in R$, $\omega \in \Omega$, we have 
	
	\qquad $u^{\omega}(x,t+1)=u^{\omega}(x,t)$  \qquad and \qquad
	$u^{\theta(s)\omega}(x,t)=u^{\omega}(x,t+s)$
\end{theorem}

In this paper, we define the global minimizer in this way: $\gamma^{\omega}:(-\infty,\infty) \to M$ is a global minimizer if and only if for any $s<t$,  $\gamma^{\omega}|[s,t]$ is minimizer for the Action $A^{\omega}(x,s;y,t)$.

 We prove that the set of global minimzer is nonempty for all $\omega \in \Omega$. Besides, under some situations, we can know the structure of ergodic invariant measures on the space of global minimizers.

$\mathbf{Acknowledgements}$ I thank Piermarco Cannarsa and Alfonso Sorrentino for their generous help. They spent much time talking with me for this project and they provided some insightful advice and helpful references.

\section{Stochastic Time-Periodic Tonelli Lagrangian}

In this section, we recall some basic notions, give the definiton of our investigated objects and examples. Throughout this paper, $(M,g)$ is a compact connected Riemannian manifold and $TM$ is its tangent bundle. John Mather originally considered the Time-Periodic Tonelli Lagrangian.
\begin{definition}[Time-Periodic Tonelli Lagrangian]
	On $M$,  a $C^2$ map $L:TM \times R \to R$ is a Time-Periodic Tonelli Lagrangian if $L$ satisfies:
	
	(I)Periodicity: $L(x,v,t+1)=L(x,v,t)$,  $\forall x \in M$,$v \in T_{x}M$ $t \in R$.
	
	(II)Convexity: For all $x \in M$,$v \in T_xM$, the Hessian matrix $\frac{\partial^2 L}{\partial v_i \partial v_j}(x,v,t)$ is positive definite.
	
	(III)Superlinearity:
	\[
	\lim_{||v|| \to +\infty}\frac{L(x,v,t)}{||v||}=+\infty
	\]
	
	uniformly on $x \in M$,$t \in R$.
	
	(IV)Completeness: The maximal solutions of the Euler-Lagrangian, that in local coordinates is:
	\[
	\frac{d}{dt}\frac{\partial L}{\partial v}(x,\dot{x},t)=\frac{\partial L}{\partial x}(x,\dot{x},t)
	\]
	
	are defined on all $R$.
\end{definition}

Following this definition from John Mather, we introduce the definition of Stochastic Time-Periodic Tonelli Lagrangian. $(\Omega,\mathscr{F},\mathbb{P})$ is a probability space, $\mathscr{F}$ is the  $\sigma-$algebra of $\Omega$.

\begin{definition}[Stochastic Time-Periodic Tonelli Lagrangian]
	On $M$,a map $L:TM \times R \times \Omega \to R$ is a measurable function. It is called Stochastic Time-Periodic Tonelli Lagrangian if $L$ satisfies:
	
	(I)Fix each $\omega \in \Omega$ ,  $L(\cdot,\cdot,\cdot,\Omega)$ is a Time-Periodic Tonelli Lagrangian on $TM \times T$.
	
	(II)Fix each $(x,v,t) \in TM \times T$, $L(x,v,t,\cdot)$ is a measurable function on $\Omega$.
\end{definition}

Next, we introduce a skew-product dynamical system. Let $\mathscr{B}$ be the Borel algebra of $\mathbb{R}$.Then the $\sigma-$algebra of $\mathbb{R}\times\Omega$ is given by $\mathscr{B}\otimes \mathscr{F}$.
\begin{definition}
	A skew-product dynamical system is a measurable map $\theta$: $R \times \Omega \to \Omega$,satisfying
	
	(I) $\forall x \in \Omega$,$\forall s,t \in R$, $\theta(0,x)=x$, $\theta(s,\theta(t,x))=\theta(s+t,x)$.
	
	(II)Fix  $t \in R$,  $\theta(t): \Omega \to \Omega$ defined as $\theta(t)(x)=\theta(t,x)$, is measure-preserving, i.e $\forall E \in \mathscr{F}$,  $\mathbb{P}({\theta(t)}^{-1}(E))=\mathbb{P}(E)$.
\end{definition}
From Definition 2.3, $ { \{\theta(t)\} }_{t \in \mathbb{R}} $ is a group of measure-preserving endomorphism of $\Omega$. In this paper, we consider the a class of Stochastic Time-Periodic Lagrangians which match the skew-product dynamical system, i.e. satisfying the following assumption:

\begin{assumption}
	$L:TM \times T \times \Omega \to R$ is a Stochastic Time-Periodic Tonelli Lagrangian, we assume that $L$ matches the skew-product dynamical system, if for $(x,v) \in TM$, $s \in R$, $\omega \in \Omega$, we have 
	\[
	L(x,v,t+s,\omega)=L(x,v,t,\theta(s)\omega)
	\]
\end{assumption}

The next two examples are constructed to show existence of Lagrangians which satisfy Assumption 2.1

\begin{example}
	$\Omega=[0,1]$ is a  probability space with Lebesgue measure.
	$f:[0,1] \to [0,1]$ is a one-to-one measurable function. Precisely,
	\[
	f(x) =
	\begin{cases}
	x & \text{if $0 \le x < \frac{1}{3}$,}\\
	x+\frac{1}{3} & \text{if $\frac{1}{3} \le x < \frac{2}{3}$,}\\
	x-\frac{1}{3} & \text{if $\frac{2}{3} \le x <1$.}
	\end{cases}
	\]
	
	$\theta: R \times \Omega \to \Omega$, $t \in R$, is defined by
	\[
	\theta(t,\omega)=f^{-1}(t+f(\omega))
	\]
	We see that $\theta:R \times \Omega \to \Omega$ is a skew-product dynamical system.
	Then define $L:TM \times T \times \Omega \to R$ by 
	\[
	L(x,v,t,\omega)=\frac{1}{2}g_{x}(v,v)+h(t+f(\omega))
	\]
	where $g_x$ is a Riemannian metric over $M$, $h$ is a periodic function from $R$ to $R$, with period $1$, $C^2$ regularity. $L$ is a Stochastic Time-Periodic Tonelli Lagrangian which satisfies Assumption 2.1
\end{example}

\begin{example}
	$\Omega=T^d$ is a  probability space with Lebesgue Measure. $T^d$ can be decomposed as 
	\[
	T^d=\prod_{k=1}^{d}([0,\frac{1}{n})\bigcup [\frac{1}{n},\frac{2}{n})\cdot\cdot\cdot\bigcup[\frac{n-1}{n},1))
	\]	
	
	Assume that $f$ is a permutation of these $n^d$ cubics by linear maps, so $f$ is a one-to-one measurable map from $T^d$ to $T^d$. Choose $(1,\alpha_2,..\alpha_d) \in R^d$ as a rationally independent vectors, i.e, there is no nonzero vector $(k_1,k_2,...k_d) \in Z^d$ such that $k_1+k_2\alpha_2+k_3\alpha_3+\cdot \cdot\cdot +k_d\alpha_d=0$
	
	Define $\phi(t):\Omega \to \Omega$  for $\forall x\in T^d$, $x=(x_1,...,x_d) \in R^d$,$t\in R$,  by
	\[
	\phi(t)(x_1,...,x_d)=(x_1+t,x_2+\alpha_2 t,..,x_d+\alpha_d t)(mod Z^d)
	\]
	For $t \in R$,  $\phi(t)$ is ergodic over $\Omega$ with stationary property: $\phi(t)\cdot\phi(s)=\phi(t+s)$
	
	Define the skew-product dynamical system  $\theta: R \times \Omega \to \Omega$, for each $t \in R$,  each $\omega \in \Omega$,
	\[
	\theta(t,x)=f^{-1}(\phi(t)\cdot (f(x))))
	\]
	
	In particular, $\{\theta(t)|t \in R\}$ is ergodic on $\Omega$.
	
	Define $L:TM \times T \times \Omega$ as follows:
	\[
	L(x,v,t,\omega)=\frac{1}{2}g_x(v,v)+h(t+\pi(\omega))
	\]
	where  $\pi:T^d \to R$ defined as $\pi(x_1,x_2,...x_d)=x_1$, $h$ is a 1-periodic function on $R$ and has regularity $C^2$, $g_x$ is the Riemannian Metric over $M$.
	Hence, $L$ is a Stochastic Time-Periodic Tonelli Lagrangian and satisfies the Assumption 2.1 For the translation on Torus, see \cite{katok1997introduction}.
\end{example}

\begin{remark}
	In Example 1, Example 2, the one-to-one measurable map $f$ is not unique.
\end{remark}

\section{Probability Measures on Metric Spaces}

In this section, we introduce some tools from Probability Theory.

$(X,d)$ is a metric space.Consider the functional space. $C_b(X)=\{ f: X \to R: f $ is continuous and bounded $\}$. Each $f \in C_b(X)$ is integrable with respect to any finite Borel measure on $X$. We introduce the notion of weak convergence of probability measure on $X$.
\begin{definition}
	$\mu,\mu_1,\mu_2,...$ are finite Borel measures on $X$. We say that $(\mu_i)_i$ converges weakly to $\mu$, if for all $f \in C_b(X)$, we have
	\[
	\lim_{i \to \infty} \int_{X}f d\mu_i \to \int_{X}f d\mu 
	\]
	we denote $\mu_i \rightharpoonup \mu$.
\end{definition}

Next, we discuss the Prokhorov metric on $(X,d)$. Denote $\mathcal{P}=\{\mu| \mu$ is a probability measure on $X \}$. $\mathcal{B}(X)$ is the set of all Borel algebra generated by open sets on $X$. The Prokhorov metric arises from the distance on $\mathcal{P}$, defined as
\begin{definition}
	For $\mu,\nu \in \mathcal{P}$, $d_p(\mu,\nu)$ is defined by
	\[
	d_P(\mu,\nu)=\inf\{\alpha>0: \mu(A) \le \nu(A_{\alpha}) +\alpha, \nu(A) \le \mu(A_{\alpha})+\alpha, \forall A \in \mathcal{B}(X) \}
	\]
	
	where $A_{\alpha}=\{x \in X: d(x,A)< \alpha \}$ if $A \neq \emptyset$, $\emptyset_{\alpha}:=\emptyset$ for all $\alpha > 0$.
\end{definition}

Then we have:
\begin{lemma}
	$(X,d)$ is the metric space, $d_{P}$ defined above,
	
	(1)$d_P$ is a metric on $\mathcal{P}=\mathcal{P}(X)$.
	
	(2)If $\mu, \mu_1, \mu_2,...\in \mathcal{P}$, $\lim_{i \to \infty}d_P(\mu_i, \mu)=0$ implies $\mu_i \rightharpoonup \mu$.
\end{lemma}

\begin{lemma}
	If $(X,d)$ is a separable metric space, then for any $\mu,\mu_1,\mu_2,...\in \mathcal{P}(X)$ one has 
	$\mu_i \rightharpoonup \mu$ if and only if $\lim_{i \to \infty}d_P(\mu,\mu_i)=0$.
\end{lemma}

\begin{lemma}
	$(X,d)$ is a separable metric space, then $\mathcal{P}=\mathcal{P}(X)$ with the Prokhorov metric $d_P$ is separable. 
\end{lemma}

\begin{lemma}
    	$(X,d)$ is a separable complete metric space, then $\mathcal{P}=\mathcal{P}(X)$ with the Prokhorov metric $d_P$ is complete. 
\end{lemma}

The proof from Lemma 3.1 to Lemma 3.4 can be checked in \cite{billingsley2013convergence}

\begin{theorem}
	Let $X$ and $Y$ be two Polish spaces and $\lambda$ be a Borel probability measure on $X\times Y$.Let us set $\mu=\pi_{X}\lambda$, where $\pi_{X}$ is the standard projection from $X \times Y$ onto $X$. Then there exists a $\mu-$almost everywhere uniquely determined family of Borel probability measures $(\lambda_x)$ on $Y$ such that 
	
	1.The function $x \to \lambda_x$ is Borel measurable, in the sense that $x \to \lambda_x(B)$ is a Borel-measurable function for each Borel-measurable set $B \in Y$.
	
	2. For every Borel-measurable function $f:X \times Y \to [0,\infty)$,
	\[
	\int_{X\times Y}f(x,y)d\lambda(x,y)=\int_{X}\int_{Y}f(x,y)d\lambda_x(y)d\nu(x)
	\]
\end{theorem}

The disintegration of measure can be seen in \cite{ambrosio2008transport}

\section{ Mane-Critical Value}
Given a Time-Periodic Tonelli Lagrangian $L:TM \times T \to R$, by Euler-Lagrange equation,i.e,
\[
\frac{d}{dt}\frac{dL}{d\dot{x}}(x,\dot{x},t)=\frac{dL}{dx}(x,\dot{x},t)
\]
we can define a time-dependent Lagrangian flow $\Phi_{s,t}:TM \times \{s\} \to TM \times \{t\}$.see \cite{mather1991action}

Denote: $\mathcal{M}_{inv}=\{\mu: \mu$ is a Borel probability over $TM \times T$, $\mu $ is invariant under the flow $\Phi_{s,t}, \forall s,t \in R \}$, the Mane-critical value of $L$ is defined as 
\begin{equation}
-c[0]=\min_{\mu \in \mathcal{M}_{inv} }\int_{TM \times T}L(x,v,t)\,d\mu 
\end{equation}

The corresponding Hamiltonian $H: T^{*}M \times T \to R$ is defined by 
\[
H(x,p,t)=\sup_{v \in TM}p(v)-L(x,v,t)
\]

Considering the Hamilton-Jacobi equation, $u: M \times T \to R$.
\begin{equation}
\partial_t u(x,t)+ H(x,\partial_x u(x,t), t)=c[0]
\end{equation}

If $u$ is a subsolution of (2), that means for each $(x_0,t_0) \in M \times T $, there exists a $C^1$ function $\phi:M \times T \to R$, $\phi \ge u$, $\phi(x_0,t_0)=u(x_0,t_0)$, we have
\[
\partial_t \phi(x_0,t_0)+ H(x_0,\partial_x \phi(x_0 ,t_0),t_0) \le c[0]
\]

For all $n \in N$, $h_n:(M \times T) \times (M \times T) \to R$ is defined as :
\[
h_n((x,t),(y,s))=\min_{\gamma \in \Sigma(x,t;y,s+n)}\int^{s+n}_{t}L(\gamma,\dot{\gamma},t)dt+nc[0]
\]

The Peierls barrier is then defined as: for $(x,t)\times (y,s) \in (M\times T) \times (M \times T)$:
\[
h((x,t),(y,s))=\liminf_{n \to \infty} h_n((x,t),(y,s))
\]

The Projected Aubry set is 
\[
\mathcal{A}_{0}:=\{(x,t)\in M \times T: h((x,t),(x,t)=0)\}
\]


In \cite{massart2006subsolutions}, Daniel Massart proved a useful theorem that we desire, we present it here:
\begin{theorem}
	There exists a $C^{1}$ critical subsolution of Hamilton-Jacobi equation which is strict at every point of $\mathcal{A}^{c}_{0}$.
\end{theorem}

To prove theorem 1.1, we introduce the notion of closed measure.

\begin{definition}
	A probability measure $\mu$ on $TM \times T$ is called closed if 
	\[
	\int_{TM \times T}|v| \, d\mu(x,v,t) < \infty
	\]
	and for every smooth function $f$ on $TM \times T$, we have 
	\[
	\int_{TM \times T}df(x,t)(v,1)d\mu(x,v,t)=0
	\]
	
	We denote the set of closed measures on $TM \times T$ as $\mathcal{M}_{c}$
\end{definition}

In \cite{massart2006subsolutions}, Daniel Massart gives a desired theorem as follows:

\begin{theorem}
	For a Time-Periodic Tonelli Lagrangian $L: TM \times T \to R$, its Mane-Critical Value can be formulated as
	\begin{equation}
	-c[0]=\min_{\mu \in \mathcal{M}_{f}}\int_{TM \times T}L(x,v,t)\,d\mu
	\end{equation}
\end{theorem}

\begin{lemma}
For a Time-Periodic Tonelli Lagrangian $L:TM \times T \to R$, if one measure $\mu \in P(TM)$ satisfies that  
\[
-c[0]=\int_{TM \times T} L(x,v,t)d\mu
\]
, $supp(\mu)$ is a compact subset in $TM \times T$.
\end{lemma}
\begin{proof}
	For each closed measure $\mu$, if $f:M \times T \to R$ is a smooth function, we have 
	\[
	\int_{TM \times T}df(x,t)(v,1)d\mu(x,v,t)=0
	\]
	If $g:M \times T \to R$ is $C^1$, we can approximate it in the uniform $C^1$ topology by a sequence of $C^{\infty}$ functions $f_n: M\times T \to R$. In particular, there is a constant $K < \infty$, for each $x \in M$, and $n \in N$, we have $|df(x,t) \cdot(v,1)| \le K(||v||+1)$.  Since $\int_{TM \times T}||v||d\mu(x.v.t)<\infty $, $df_n(x,t)(v,1) \to dg(x,t)(v,1)$ by the dominated convergence theorem, we obtain that $\int_{TM \times T}dg(x,t)\cdot(v,1) d\mu(x,v,t)=0$.
	
 Suppose that a closed measure $\mu$ does satisfy that $\int_{TM \times T}L(x,v,t)d\mu=-c[0]$, From Theorem 4.1, we know that there exists $u: M \times T \to R$ as a $C^1$ critical subsolution, such that for $x \in \mathcal{A}^{c}_{0}$, we have
	\[
	\partial_t u(x,t)+ H(x,\partial_x u(x,t),t)< c[0]
	\]
We integrate the following equation:
\[
\partial_x u(x,t)\cdot v + \partial_t u(x,t) \le L(x,v,t)+ H(x,\partial_x u(x,t),t)+ \partial_tu(x,t) \le L(x,v,t)+c[0]
\] 	
Then we get 
	\[
	0 \le \int_{TM \times T}L(x,v,t)+H(x,d_x u,t) d\mu \le 0
	\]
	So we know that if $(x,v)\in supp(\mu)$, we have 
	\begin{equation}\nonumber
	\begin{aligned}
	&\partial_t u(x,t)+H(x,\partial_x u(x,t),t)=c[0]\\
	&\partial_x u(x,t)(v)=L(x,v,t)+H(x,\partial_x u(x,t) ,t)
	\end{aligned}
	\end{equation}
	So we know that $x \in \mathcal{A}_{0}$, $\partial_x u(x,t)=\frac{\partial L}{\partial v}(x,v,t)$, and  $v=\frac{\partial H}{\partial p}(x,\partial_x u(x,t),t)$. Hence we conclude that $supp(\mu)$ is compact.
	
\end{proof}

Coming back to a Stochastic Time-Periodic Tonelli Lagrangian $L: TM \times T \times \Omega \to R$, for each $\omega \to \Omega$, $L(\cdot,\cdot,\cdot,\omega)$ is a Time-Periodic Tonelli Lagrangian. So $L(\cdot,\cdot,\cdot,\omega)$ has a Mane-Critical Value, we denote it as $\alpha(\omega)$.

\begin{lemma}
	Given a compact connected Riemannian manifold $(M,g)$, the manifold $M$, its tangent bundle $TM$ and $TM \times T$ are separable, complete metric spaces.
\end{lemma}
\begin{proof}{First proof of Theorem 1.1}
	
	Since $M$ is a compact manifold, we can find a finite number of charts $\{(U_i,\phi_i)| 1 \le i \le N\}$ to cover $M$.For $1 \le i \le N$, $U_i$ is isomorphic to a open subset of $R^n$, so $U_i$ is a separable open set for $1 \le i \le N$. Hence, $M=\cup_{1 \le i \le N}U_i$ is also separable. 
	
	$TU_i$ is isomorphic to $\phi_i(U_i) \times R^n$, which is separable. So $TU_i$ is separable. Then $TM \subset \cup^{N}_{i=1}TU_i$. Hence $TM$ is separable.
	
	To prove $M$ is a metric space, we define a distance on $M$. For any $x,y \in M$, $\Sigma(x,y)=\{\gamma|\gamma :[0,1] \to M, \gamma$ is absolutely continuous on $M\}$. We denote the distance between $x$ and $y$ as $d(x,y)$, defined by
	\[
	d(x,y)=\inf_{\gamma \in \Sigma(x,y)}\int^{1}_{0}g_{\gamma(t)}(\dot{\gamma}(t),\dot{\gamma}(t))dt
	\]
	
	$M$ is connected, any two different points can be connected by an absolutely continuous curve. $d_{M}(x,y)$ defines a metric on $M$.(see \cite{petersen2006riemannian}). Since any compact metric space is complete, we can know that $M$ is complete.
	
	To prove $TM$ is a metric space, firstly we prove $TM$ is a Riemannian manifold.
	
	We construct the Riemannian metric locally on $TM$. $\{(U_i,\phi_i)|1 \le i \le N \}$ are charts which cover $M$, then $\{(TU_i,d\phi_i)|1 \le i \le N\}$ are charts which cover $TM$, and $d {\phi}_i( TU_i)={\phi}_j(U_i)\times R^n, 1 \le i \le N$. For some $j$,  $1 \le j \le N$,  $\xi \in TU_j$,  $(x,v)=d \phi(\xi) \in \phi(U_j) \times R^n $, for $V,W \in T_{\xi}(TU_j)$, we write $d\phi_j(V)=(V_1,V_2) \in R^{n+n}$,  $d\phi_j(W)=(W_1,W_2) \in R^{n+n}$. $\pi:TM \to M$ is the canonical projection. We define the Riemannian Metric on $TM$ as:
	\[
	G_{\xi}(V,W)=g_{\pi(\xi)}((d\phi_j)^{-1}V_1,(d\phi_j)^{-1}W_1)+g_{\pi(\xi)}((d\phi_j)^{-1}V_2,(d\phi_j)^{-1}W_2)
	\]
	This Riemannian metric on $TM$ is well defined. Then we can prove than $\mathbb{TM}$ is a metric space by viewing it as a connected Riemannian Manifold, we denote the metric by $d_{TM}(\cdot,\cdot)$.
	
	To prove $TM$ is complete under this metric.Let $\xi_1,\xi_2,..,\xi_n,...$ be a cauchy sequnce in $TM$ under the metric $d_{TM}(\cdot,\cdot)$. Then $\pi(\xi_1),\pi(\xi_2),...\pi(\xi_n),..$ are cauchy sequences in $M$ under the metric $d(\cdot,\cdot)$, since $d(\pi(p),\pi(q)) \le d_{TM}(p,q)$ for all $p,q \in TM$.Then we can find $x_0 \in TM$ such that $\lim_{i \to \infty}d(\pi(\xi_i),x_0)=0$.Then there is a local chart $(U,\phi)$ on $M$ such that $x_0 \in U$.There exists an integer $N$, such that for $i \ge N$,$\pi(\xi_i) \in U$. If we write $d\phi(\xi_i)=(x_i,v_i) \in \phi(U_i)$.We have 
	$\lim_{i \to \infty}x_i=x_0$.Since $(x_i,v_i) \in d\phi(U_i)\times R^n$ is cauchy sequnce,then $v_i \in R^n, i \ge N$ is a cauchy sequence. $R^n$ is complete, so there exists $v_0 \in R^n$, such that $\lim_{i \to \infty}v_i=v_0$. Then $\lim_{i \to \infty}(x_i,v_i)=(x_0,v_0)$. So $TM$ is complete.

     Both $TM$ and $T$ are separable, complete metric space, so we can define the metric $d_{TM \times T}$ in this way , if $(\xi_1,t_1),(\xi_2,t_2) \in TM \times T$, $d_{TM \times T}((\xi_1,t_1),(\xi_2,t_2))=d_{TM}(\xi_1,\xi_2)+d_{T}(t_1,t_2)$. It is clear that $TM \times T$ is also a complete separable metric space. 
\end{proof}

We turn to the first version of proof of Theorem 1.1 
\begin{proof}
From Lemma 4.1, $TM \times T$ is a complete separable metric space. With Lemma 3.3, Lemma 3.4, $\mathcal{P}(TM \times T)$ with the Prokhorov metric $d_{P}$ is a complete separable metric space. The closed measure is a subset of $\mathcal{P}(TM \times T)$, so we can pick up a countable dense subset of closed measure $\{\mu_k|k=1,2,...,n,...\}$. Fix $\omega$, $L(\cdot,\cdot,\cdot,\omega)$ is a Time-Periodic Tonelli Lagrangian, From the lemma 4.2 and lemma 3.4, we know there is a $\mu \in P(TM \times T)$ making (1) holds, and $\mu$ is indeed a closed measure,  (see \cite{massart2006subsolutions},\cite{fathi2004existence}). From Lemma 6.1, we know that the support of $\mu$ is compact,There exists $R(\omega)>0$ such that if $(x,v) \in supp(\mu)$,we have $g_x(v) \le R(\omega)$. We can find a smooth function $\chi_{\omega}: [0,\infty) \to [0,1] $ such that $\chi(x)=1$ when $ 0 \le x \le R(\omega)$; $\chi(x)=0$ when $x \ge R(\omega)+1$. We have
\[
-\alpha(\omega)=\inf_{\mu \in {TM \times T}}\int_{\mu \in \mathcal M}L(x,v,t,\omega)\chi_{\omega}(g_x(v))d\mu(x,v,t)
\]
Since $L(x,v,t,\omega)\chi_{\omega}(g_x(v))$ is a continuos bounded function, so we can apply weak convergence of probability measure on $L(x,v,t,\omega)\chi_{\omega}(g_x(v))$.Since $\{\mu_k|k=1,2,...\}$ is dense on $\mathcal{P}(TM \times T)$, we know that
\[
-\alpha(\omega)=\inf_{k \in N} \int_{TM \times T}L(x,v,t,\omega)\chi_{\omega}(g_x(v))d\mu_{k}=\inf_{k \in N} \int_{TM \times T}L(x,v,t,\omega)d\mu_k
\]
By Fubini Theorem (see \cite{cannarsa2015introduction}), for each closed measure $\mu$, we know that $\int_{TM \times T}L(x,v,t,\omega)d\mu$ is measurable function on $\Omega$. Since the infimum of a countable measurable funtion is measurable on $\Omega$, $\alpha(\omega)$ is measurable. Since $L(x,v,t,\theta(s)\omega)=L(x,v,t+s,\omega)$,for $s \in R$, we know that $\alpha(\theta(s)\omega)=\alpha(\omega)$. If $\{\theta(s)$,$s \in R\}$ is ergodic on $\Omega$, we know that $\alpha(\omega)$ is constant almost everywhere.This finishes the proof of Theorem 1.1
\end{proof}

The author has an another simple proof, we take advantage of a useful result from \cite{contreras2013weak}
\begin{proposition}
If $L:TM \times T \to R$ is a Tonelli Lagrangian, the  Mane-Critical Value has another interpretation:

$\alpha(0)=\min\{k: \int L+k \ge 0$ for all closed curves $\gamma\}$

A curve $\gamma:[a,b] \to M$ is called closed if $\gamma(a)=\gamma(b)$ and $b-a$ is an integer.

\end{proposition}

\begin{proof}{second proof of Theorem 1.1}
	
For integers $m<n$, Let $C^{0}([m,n], TM)$ denote the space of continuous functions from $[m,n]$ to $TM$ with the  bounded uniform norm. $M$ is a compact manifold, by whitney's theorem, there exists $m \in N$, such that $M$ can be embedded isomorphically into $R^m$. Therefore, $TM$ can be embedded as a submanifold of $R^{2m}$. By Stone-Weierstrass theorem, the space of continuous functions from $[m,n]$ to $R$ is separable. Therefore, the space of continuous functions from $[m,n]$ to $R^{2m}$ is separable. As its subset, $C^{0}([m,n],TM)$ is separable. So $\bigcup_{m<n}C^{0}([m,n],TM)$ is separable. For a Tonelli Lagrangian, the closed curve which integral of Lagrangian attains the minimum along it, satisfies the Euler-Lagrangian equation, and is $C^2$ and compact. So the space of curves, along which Lagrangian Action achieves the minimum, is a subspace of $\bigcup_{m<n}C^{0}([m,n],TM)$.So we can find a sequence of closed curves $\{\gamma_{i}\}$ for $i \in N$, such that 
\[
\alpha(\omega)=\inf_{i \in N}\int_{\gamma_i}L(\gamma_i(t),\dot{\gamma_i}(t),t,\omega)dt
\]
By Fubini theorem, for each $i \in N$, $\int_{\gamma_i}L(\gamma_i(t),\dot{\gamma_i}(t),t,\omega)dt$ is a measurable function. As the infinimun of countable measurable functions, we know that $\alpha(\omega)$ is a measurable function. Since $L(x,v,t,\theta(s)\omega)=L(x,v,t+s,\omega)$, for $s \in R$, we get $\alpha(\theta(s)\omega)=\alpha(\omega)$. If $\{\theta(s),s \in R\}$ is ergodic on $\Omega$, we know that $\alpha(\omega)$ is constant almost everywhere. This finishes the proof of Theorem 1.1.

\end{proof}

\section{Semiconcave Estimates on Lagrangain Action}

In this section we have a revision of some basic properties of a class of nonsmooth functions, the so-called semiconcave functions. The good properties of semiconcave functions provide fundamental technical tools for the analysis of singularities of Lagrangian Action and Weak KAM Solutions.

It is well known that a real-valued function $u$ is semiconcave in an open demain $U \in R^n$ if, for any compact set $K \subset U$, there exists a constant $C \in R$ such that
\[
tu(x_1)+(1-t)u(x_0)- u(tx_1+(1-t)x_0) \le Ct(1-t){|x_1-x_0|}^2
\]
for all $t \in [0,1]$ and for all $x_0,x_1 \in K$ satisfying $[x_0,x_1] \subset K$. We refer to such a constant $C$ as a semiconcavity constant for $u$ on $K$. We denote by $SC(U)$ the class of all semiconcave functions in $U$.

We  review some differentiability properties of semiconcave functions. To begin, let us recall that any $u \in SC(U)$ is locally Lipschitz continuous.(see \cite{cannarsa2004semiconcave}).Hence,by Rademacher's Theorem, $u$ is differentiable a.e in $U$ and the gradient of $u$ is locally bounded. Then, the set
\[
D^{*}_{x}=\{p \in R^n: U \ni x_i \to x, Du(x_i) \to p\}
\]
is nonempty for any $x \in U$. The elements of $D^{*}u(x)$ are called reachable gradients.

The superdifferential of any function $u:U \to R$ at a point $x \in U$ is defined as 

\[
D^{+}_{x}u=\{p \in R^n: \limsup_{h \to 0}\frac{u(x+h)-u(x)-<p,h>}{|h|} \le 0\}
\]
Similarly, the subdifferential of $u$ at $x$ is given by 
\[
D^{-}_{x}u=\{p \in R^n: \liminf_{h \to 0}\frac{u(x+h)-u(x)-<p,h>}{|h|} \ge 0\}
\]
Next, we list some properties:
\begin{proposition}
	Let $u:A \to R$ and $x \in A$. Then the following properties hold true.
	
	(1)$D^{+}_{x}u$ and $D^{-}_{x}u$ are closed convex sets.
	
	(2)$D^{+}_{x}u$ and $D^{-}_{x}u$ are both nonempty if and only if $u$ is differentiable at $x$; in this case we have that 
	\[
	D^{+}_{x}u=D^{-}_{x}u=\{D_{x}u\}
	\]
	
	Furthermore, when $u: U \to R$ be a semiconcave function, we have
	
	(3)$D^{+}_{x}u=cov D^{*}_{x}u$
	
	(4)$D^{+}_{x}u \neq \emptyset $
	
	(5)When $D^{+}_{x}u$ is a singleton,  $u$ is differentiable at $x$.
\end{proposition}
The proof of Proposition 5.1 can be seen in \cite{cannarsa2004semiconcave}

Given $L: TM \times T \to R $ a Time-Periodic Tonelli Lagrangian. For an absolutely continuous curves $\gamma:[s,t] \to M$, the action of $L$ along $\gamma$ is defined as $A(\gamma)=\int^{t}_{s}L(\gamma(u),\dot{\gamma}(u),u)du$.

$\Sigma(s,y;t,x)=\{\gamma: \gamma:[s,t] \to M $ is absolutely continuous , and $\gamma(s)=y,\gamma(t)=x\}$. The Lagrangian action $A(s,y;t,x)$ is defined as
\[
A(s,y;t,x)=\min_{\gamma \in \Sigma(s,y;t,x)}\int^{t}_{s}L(\gamma(u),\dot{\gamma}(u),u)du
\]

If $\gamma:[s,t] \to M$ attains the minimum of $A(s,y;t,x)$, then from variational methods, we know that $\gamma$ satisfies Euler-Lagrange equation and $\gamma$ is $C^2$.

\begin{lemma}
	Fix $s_1<t_1$, for any $s\le s_1,t \ge t_1$, the lagrangian Action $A(s,\cdot;t,\cdot)$ is equi-semiconcave on $M \times M$, therefore equi-Lipschitz.
\end{lemma}

\begin{proof}
	To give a proof, we use the variational methods.Fix $\forall x \in M$, we can find a chart such that $x \in U \subset M$.  Without loss of generality, we assume that $U$ is a open ball of $R^n$. If $\gamma \in \sum_{m}(s,y;t,x+v)$, we can find $h>0$,such that $\gamma([t-h,t]) \in U$, we can find a Ball $B(0,r) \in R^n$, such that for $v \in B(0,r)$, $s \in [0,h]$, we have $\gamma(t-h+s)+\frac{s}{h}v \in U$. Fix $v \in B(0,r)$, we define $\gamma_h \in \sum(s,y;t,x)$ in the following way:
	when $s \le u \le t-h$,$\gamma_h(u)=\gamma(u)$;
	 when $t-h \le u \le t$,$\gamma_h(u)=\gamma(u)+\frac{u+h-t}{h}v$. Assume $F_k=\max_{||v||_{U} \le k}||\partial_{vv}L(x,v,t)||_{U}<\infty$, $E_k=\max_{||v||_{U} \le k}||\partial_{xv}L(x,v,t)||_{U}<\infty$. From Tonelli theorem,  there exists $K(s_1,t_1)$ such that $|\dot{\gamma}(u)| \le K(s_1,t_1)$ for $s \le u \le t$ where $\gamma \in \Sigma(s,y;t,x)$ is a minimizer. We have the following estimates:
	
	\begin{equation}\nonumber
	\begin{aligned}
	&A(\frac{s_1+t_1}{2},\gamma(\frac{s_1+t_1}{2});t,x+v)-A(\frac{s_1+t_1}{2},\gamma(\frac{s_1+t_1}{2});t,x)\\
	&\le \int^{t}_{\frac{s_1+t_1}{2}}L(\gamma_h(u),\dot{\gamma_h}(u),u)du-\int^{t}_{\frac{t_1+s_1}{2}}L(\gamma(u),\dot{\gamma}(u),u)\,du\\
	&=\int^{t}_{t-h}L(\gamma_h(u),\dot{\gamma_h}(u),u)-L(\gamma(u),\dot{\gamma}(u),u)\,du\\
	&\le \int^{t}_{t-h}L(\gamma_h(u),\dot{\gamma_h}(u),u)-L(\gamma_h(u),\dot{\gamma}(u),u)+ L(\gamma_h(u),\dot{\gamma}(u),u)-L(\gamma(u),\dot{\gamma}(u),u)\,du\\
	&\le \int^{t}_{t-h}\frac{\partial L}{\partial v}(\gamma_h(u),\dot{\gamma}(u),u)\cdot \frac{1}{h}v+\frac{\partial L}{\partial x}(\gamma(u),\dot{\gamma}(u),u)\cdot \frac{u+h-t}{h}\cdot v\\
	&+ \frac{1}{2h^2}F_{K(s_1,t_1)+r}{||v||}^2_{U}+  \frac{1}{2}E_{K(s_1,t_1)}{||v||}^2_{U}\,du \\
	&\le \int^{t}_{t-h}\frac{\partial L}{\partial x}(\gamma(u),\dot{\gamma}(u),u)\cdot \frac{u+h-t}{h}\cdot v+ \frac{\partial L}{\partial v}(\gamma(u),\dot{\gamma}(u),u)\cdot \frac{1}{h}v du + \frac{h}{2} E_{K(s_1,t_1)}{||v||}^2_{U}+\\ &+\frac{1}{2 h}F_{K(s_1,t_1)+r}{||v||}^2_{U} 
	  + E_{K(s_1,t_1)}{||v||}^2_{U}\\
	&=\frac{\partial L}{\partial v}(\gamma(u),\dot{\gamma}(u),u)\frac{u+h-t}{h}\cdot v {\mid}^{t}_{t-h}
	+E_{K(s_1,t_1)}{||v||}^2_U+ \frac{1}{2h}F_{K(s_1,t_1)+r}{||v||}^2_U + \frac{h}{2}E_{K(s_1,t_1)}{||v||}^2_U\\
	&+\int^{t}_{t-h}\{-\frac{d}{dt}\frac{\partial L}{\partial x}(\gamma(u),\dot{\gamma}(u),u)+\frac{\partial L}{\partial x}(\gamma(u),\dot{\gamma}(u),u) \} \cdot \frac{u+h-t}{h} \cdot v \,du \\
	&=\frac{\partial L}{\partial v}(\gamma(t),\dot{\gamma}(t),t)\cdot v + \{\frac{h}{2}E_{K(s_1,t_1)}+\frac{1}{2h}F_{K(s_1,t_1)+r}+ E_{K(s_1,t_1)}\}{||v||}^2_U
	\end{aligned}
	\end{equation}
	
	Similarly, we have 
	\begin{equation}{\nonumber}
	\begin{aligned}
	&A(\frac{s_1+t_1}{2},\gamma(\frac{s_1+t_1}{2});t,x-v)-A(\frac{s_1+t_1}{2},\gamma(\frac{s_1+t_1}{2});t,x)\\
	&\le -\frac{\partial L}{\partial v}(\gamma(t),\dot{\gamma}(t),t)\cdot v + \{\frac{h}{2}E_{K(s_1,t_1)}+\frac{1}{2h}F_{K(s_1,t_1)+r}+ E_{K(s_1,t_1)}\}{||v||_U}^2
	\end{aligned}
	\end{equation}
	
	So we have 
	\begin{equation}{\nonumber}
	\begin{aligned}
	&A(\frac{s_1+t_1}{2},\gamma(\frac{s_1+t_1}{2});t,x+v)+A(\frac{s_1+t_1}{2},\gamma(\frac{s_1+t_1}{2});t,x-v)\\
	&-2A(\frac{s_1+t_1}{2},\gamma(\frac{s_1+t_1}{2});t,x)
	\le \{E_{K(s_1,t_1)}+\frac{1}{h}F_{K(s_1,t_1)+r}+ 2E_{K(s_1,t_1)}\}{||v||_U}^2
	\end{aligned}
	\end{equation}
	
	Therefore, there exists a constant $C(U,s_1,t_1)$ such that
	\begin{equation}{\nonumber}
	\begin{aligned}
	&A(\frac{s_1+t_1}{2},\gamma(\frac{s_1+t_1}{2});t,x+v)+A(\frac{s_1+t_1}{2},\gamma(\frac{s_1+t_1}{2});t,x-v)\\
	&-2A(\frac{s_1+t_1}{2},\gamma(\frac{s_1+t_1}{2});t,x)
	\le C(U){||v||_U}^2
	\end{aligned}
	\end{equation}
	
	On the other hand, there exists a local chart $V$ such that $y \in V$, there exists $r_{1} > 0$, such that $B(y,r_{1}) \subset V$, for any $w \in B(y,r_{1})$,we have
	\begin{equation}{\nonumber}
	\begin{aligned}
	&A(s,y+w;\frac{s_1+t_1}{2},\gamma(\frac{s_1+t_1}{2}))+A(s,y-w;\frac{s_1+t_1}{2},\gamma(\frac{s_1+t_1}{2}))\\
	&-2 A(s,y;\frac{s_1+t_1}{2},\gamma(\frac{s_1+t_1}{2}))
	\le C(V){||w||_V}^2
	\end{aligned}
	\end{equation}
	
	Finally, we integrate the results above,
	\begin{equation}{\nonumber}
	\begin{aligned}
	&A(s,y+w;t,x+v)+A(s,y-w;t,x-v)-2A(s,y;t,x)\\
	&\le A(\frac{s_1+t_1}{2},\gamma(\frac{s_1+t_1}{2});t,x+v)+A(\frac{s_1+t_1}{2},\gamma(\frac{s_1+t_1}{2});t,x-v)\\
	&-2 A(\frac{s_1+t_1}{2},\gamma(\frac{s_1+t_1}{2});t,x)+ A(s,y+w;\frac{s_1+t_1}{2},\gamma(\frac{s_1+t_1}{2}))\\
	&A(s,y-w;\frac{s_1+t_1}{2},\gamma(\frac{s_1+t_1}{2}))-2 A(s,y;\frac{s_1+t_1}{2},\gamma(\frac{s_1+t_1}{2}))\\
	&\le C(U){||v||}^2_{U}+C(V){||w||}^2_{V}
	\end{aligned}
	\end{equation}
	So $A(s,\cdot;t,\cdot)$ is locally equi-semiconcave for $s \le s_1$,$t \ge t_1$. Since $M$ is compact, 
	$A(s,\cdot,t,\cdot)$ is globally equi-semiconcave for $s \le s_1$,$ t \ge t_1$.
\end{proof}

\begin{lemma}
	For each minimizing curve $\gamma \in \Sigma(s,y;t,x)$ attaining the minimum of the action $A(s,y;t,x)$, we have 
	\[
	p(t)=\partial_v L(x,\dot{\gamma}(t),t,\omega) \in D^{+}_{x}A(s,y;t,x)
	\]
	and
	\[
	-p(s)=-\partial_v L(\gamma(s),\dot{\gamma}(s),s) \in D^{+}_{y}A(s,y;t,x)
	\]
\end{lemma}

\begin{proof}
	we find a local chart $U \subset M$, such that $x \in U$. Without loss of generality, we assume that $U$ is a open ball of $R^N$. If $\gamma \in \Sigma(s,y;t,x)$, we can find $h>0$, such that $\gamma([t-h,t]) \in U$, we can find a Ball $B(0,r)$,  for $v \in B(0,r)$,we define $\gamma_h \in \sigma(s,y;t,x+v)$   such that $\gamma_h(u)=\gamma(u)$ for $s \le u \le t-h$,and $\gamma_h(u)=\gamma(u)+\frac{u+h-t}{h}v \in U$ when $r$ is small enough. Assume $F_k=\max_{||v||_U \le k }{||\partial_{vv}L(x,v,t)||}_{U}<\infty$, $E_k=\max_{||v||_{U} \le k}{||\partial_{xv}L(x,v,t)||}_{U}<\infty$. From Tonelli Theroem, there exists $K(s,t)$ such that $|\dot{\gamma}(u)| \le K(s,t)$ for $s \le u \le t$.
	
	\begin{equation}\nonumber
	\begin{aligned}
	&A(s,y,t,x+v)-A(s,y;t,x) \\
	&=\int^{t}_{t-h}L(\gamma_h(u),\dot{\gamma_h}(u),u)-L(\gamma(u),\dot{\gamma}(u),u)\,du\\
	&\le \int^{t}_{t-h}L(\gamma_h(u),\dot{\gamma_h}(u),u)-L(\gamma_h(u),\dot{\gamma}(u),u)+ L(\gamma_h(u),\dot{\gamma}(u),u)-L(\gamma(u),\dot{\gamma}(u),u)\,du\\
	&\le \int^{t}_{t-h}\frac{\partial L}{\partial v}(\gamma_h(u),\dot{\gamma}(u),u)\cdot \frac{1}{h}v+\frac{\partial L}{\partial x}(\gamma(u),\dot{\gamma}(u),u)\cdot \frac{u+h-t}{h}\cdot v\\
	&+ \frac{1}{2h^2}F_{K(s_1,t_1)+r}{||v||}^2_{U}+  \frac{1}{2}E_{K(s_1,t_1)}{||v||}^2_{U}\,du \\
	&\le \int^{t}_{t-h}\frac{\partial L}{\partial x}(\gamma(u),\dot{\gamma}(u),u)\cdot \frac{u+h-t}{h}\cdot v+ \frac{\partial L}{\partial v}(\gamma(u),\dot{\gamma}(u),u)\cdot \frac{1}{h}v du + \frac{h}{2} E_{K(s_1,t_1)}{||v||_{U}}^2\\
	&+ \frac{1}{2 h}F_{K(s_1,t_1)+r}{||v||_{U}}^2 
	  + E_{K(s_1,t_1)}{||v||_{U}}^2 \\
	&=\frac{\partial L}{\partial v}(\gamma(u),\dot{\gamma}(u),u)\frac{u+h-t}{h}\cdot v {\mid}^{t}_{t-h}
	+E_{K(s_1,t_1)}{||v||_U}^2+ \frac{1}{2h}F_{K(s_1,t_1)+r}{||v||_U}^2 + \frac{h}{2}E_{K(s_1,t_1)}{||v||_U}^2\\
	&+\int^{t}_{t-h}\{-\frac{d}{dt}\frac{\partial L}{\partial v}(\gamma(u),\dot{\gamma}(u),u)+\frac{\partial L}{\partial x}(\gamma(u),\dot{\gamma}(u),u) \} \cdot \frac{u+h-t}{h} \cdot v \,du \\
	&=\frac{\partial L}{\partial v}(\gamma(t),\dot{\gamma}(t),t)\cdot v + \{\frac{h}{2}E_{K(s_1,t_1)}+\frac{1}{2h}F_{K(s_1,t_1)+r}+ E_{K(s_1,t_1)}\}{||v||_U}^2
	\end{aligned}
	\end{equation}
	So we have 
	\[
	\lim_{||v|| \to 0}\frac{A(s,y;t,x+v)-A(s,y;t,x)-\frac{\partial L}{\partial v}(\gamma(t),\dot{\gamma}(t),t)\cdot v}{||v||} \le 0
	\]
	This proves that $\frac{\partial L}{\partial v}(\gamma(t),\dot{\gamma}(t),t) \in D^{+}_{x}A(s,y;t,x)$.
	
	In a similar way, we can prove that $-\frac{\partial L}{\partial v}(\gamma(s),\dot{\gamma}(s),s) \in D^{+}_{y}A(s,y;t,x)$
\end{proof}

\begin{lemma}
	Fix $s <t $, $\forall p \in D^{*}_{x}A(s,y;t,x)$, there exists a minimizer $\gamma \in \Sigma(s,y;t,x)$ of $A(s,y;t,x)$ such that $ p=-\frac{\partial L}{\partial v}(\gamma(t),\dot{\gamma}(t),t)$.	
\end{lemma}

\begin{proof}
	First step, when $A(s,y;t,x)$ is differentiable at $x$, $D^{+}_{x}A(s,y;t,x)=\{D_{x}A(s,y;t,x)\}$. For any minimizer $\gamma \in \Sigma(s,y;t,x)$ of Action $A(s,y;t,x)$, $-\frac{\partial L}{\partial v}(\gamma(t),\dot{\gamma}(t),t)= D_{x}A(s,y;t,x)$. By Euler-Lagrangian equation, the minimizer of $A(s,y,t,x)$ is unique.
	
	Second step, when $A(s,y;t,x)$ is not differentiable at $x$, for any $p \in D^{*}_{x}A(s,y;t,x)$, there exists $\{x_n | n =1,2,...\} \subset M$, $A(s,y;t,x)$ is differentiable at $\{x_n, n=1,2...\}$, .and $\lim_{n \to \infty} x_n=x$, $\lim_{n \to \infty} D_{x}A(s,y,t,x_n)=p$. From the first step, we know that there exists $\gamma_n \in \Sigma(s,y,t,x_n)$, such that $-\frac{\partial L}{\partial v}(\gamma_n(t),\dot{\gamma_n}(t),t)=D_{x}A(s,y;t,x_n)$.From Euler-Lagrangian equation, there exists a Lagrangian flow $(\gamma,p):[s,t] \to TM$, such that $\gamma(t)=x$, $p(t)=p$, By continuity of dependence on the initial values of Ordinary Differential Equation, and $\lim_{n \to \infty}(\gamma_n(t),-\frac{\partial L}{\partial v}(\gamma_n(t),\dot{\gamma_n}(t),t))=(x,p)$. so $y=\gamma(s)=\lim_{n \to \infty}\gamma_n(s)$. Therefore, $\gamma \in \Sigma(s,y;t,x)$. By lower semi-continuity of the Action, we have \[\int^{t}_{s}L(\gamma(u),\dot{\gamma}(u),u)du \le \lim_{n \to \infty}\int^{t}_{s}L(\gamma_n(u),\dot{\gamma_n}(u),u)du
	\] Since $A(s,y,t,x)$ is semi-concave on $M$ and semi-concave functions are lipschitz continuous. We know that $\lim_{n \to \infty}A(s,y;t,x_n)=A(s,y,t,x)$. So $\lim_{n \to \infty}\int^{t}_{s}L(\gamma_n(u),\dot{\gamma_n}(u),u)du =A(s,y,t,x)$. So $\gamma:[s,t] \to M$ is a minimizer of $A(s,y,t,x)$. This finishes the proof.
\end{proof}

\begin{corollary}
	The three following conditions are equivalent:
	
	(1) $A(s,y;t,x)$ has only one minimizer in $\Sigma(s,y;t,x)$
	
	(2) $A(s,y;t,x)$ is differentiable at $x$.
	
	(3) $A(s,y;t,x)$ is differentiable at $y$.
\end{corollary}
\begin{proof}
	$(2),(3) \rightarrow (1)$ is the direct consequence of Lemma 5.3.
	we prove $(1) \rightarrow (2)$. If $A(s,y,t,x)$ is not differentiable at $x$, $D^{+}_{x}A(s,y,t,x)$ contains more than one point. Since  $A(s,y,t,x)$ is semi-concave, we know that $D^{+}_{x}A(s,y,t,x)$ is a convex compact set, and is the convex hull of $D^{*}_{x}A(s,y,t,x)$. So $D^{*}_{x}A(s,y,t,x)$ contains more than one point, by lemma 5.3,  $A(s,y,t,x)$ has more than one minimizer. $(1) \to (3)$ is similar.
\end{proof}

\section{Weak KAM Solution}

In this section, we use Lax-Oleinik operator to construct a class of Weak KAM Solutions and prove that they are measurable over $\Omega$.

\begin{notation}
	$\Sigma(s,y;t,x)$ is the set of absolutely continuous curves $\gamma:[s,t] \to M $ such that $\gamma(s)=y$ and $\gamma(t)=x$.

$\Sigma^{\omega}_{m}	(s,y;t,x)$ denotes the set of the minimizers for the Action $A^{\omega}(s,y;t,x)$

	$D=C_{0}(M,R)$ is the real-valued continuous function space over $M$ with the uniform topology. $\mathscr{D}$ is the Borel algebra generated by open sets of $D$.
	
	$C_{0}(M \times R,R)$ is the real-valued continuous function space over $M \times R$.
	
	$f(\omega)=\sup_{x \in M, t \in [0,1]}|L(x,0,s,\omega)|$ is finite since $M$ is compact.
	
	$C(\omega)=\sup\{|L(x,v,t)|$ for $x \in M, |v| \le dist(M),t \in [0,1]\}$.
\end{notation}

\begin{definition}	
	1. For $\lambda \in R$, $ u \in C_{0}(M,R)$, the operator $T^{\omega}_{\lambda}: C_{0}(M,R) \to C_{0}(M \times R,R)$ is defined as
	\begin{equation}\nonumber
	T^{\omega}_{\lambda}(u)(x,t)=\min_{y\in M}\{u(y)+ A^{\omega}(t-\lambda,y;t,x)+\lambda\alpha(\omega)\}
	\end{equation}
	
	2.For $u \in C_{0}(M,R)$, $u^{\omega}(x,t)$ is defined as
	\begin{equation}\nonumber
	u^{\omega}(x,t)=\liminf_{\lambda\to +\infty}{T^{\omega}_{\lambda}(u)(x,t)}
	\end{equation}
	
\end{definition}

3.For $\lambda \in R$,$ u \in C_{0}(M,R)$, the operator $T^{\omega}_{\lambda,+}: C_{0}(M,R) \to C_{0}(M \times R,R)$ is defined as
\begin{equation}\nonumber
T^{\omega}_{\lambda,+}(u)(x,t)=\min_{y\in M}\{u(y)+ A^{\omega}(t,x;t+\lambda,y)+\lambda\alpha(\omega)\}
\end{equation}

4.For $u \in C_{0}(M,R)$, $u^{\omega}_{+}(x,t)$ is defined as
\begin{equation}\nonumber
u^{\omega}_{+}(x,t)=\liminf_{\lambda\to +\infty}{T^{\omega}_{\lambda,+}(u)(x,t)}
\end{equation}

\begin{lemma}
	When $\omega \in \Omega$. Fix $t \in R$, $x \in M$, $u \in C(M,R)$, the Lax-Oleinik Operator $T^{\omega}_{\lambda}u(x,t)$ is Lipschitz when $\lambda \ge 0$ with Lipschitz constant $|f(\omega)|+|\alpha(\omega)|$.
\end{lemma}
\begin{proof}
	For $\lambda_1>\lambda_2\ge 0$, $\forall \epsilon>0$, we can find $y_1,y_2 \in M$, such that 
	\begin{align}{\nonumber}
		T^{\omega}_{\lambda_1}(u)(x,t)+\epsilon=u(y_1)+A^{\omega}(t-\lambda_1,y_1;t,x)+ \lambda_1\alpha(\omega)\\{\nonumber}
		T^{\omega}_{\lambda_2}(u)(x,t)+\epsilon=u(y_2)+A^{\omega}(t-\lambda_2,y_1;t,x)+\lambda_2\alpha(\omega){\nonumber}
	\end{align}
	then 
	\begin{equation}\nonumber
	\begin{aligned}
		&T^{\omega}_{\lambda_1}(u)(x,t)\le u(y_2)+A^{\omega}(t-\lambda_2,y_2;t,x)+\lambda_2\alpha(\omega)+
		\int^{t-\lambda_2}_{t-\lambda_1}L(y_2,0,s,\omega)ds\\
		&+ (\lambda_1-\lambda_2)\alpha(\omega)
		\le T^{\omega}_{\lambda_2}(u)(x,t)+\epsilon+(\lambda_1-\lambda_2)(|\alpha(\omega)|+|f(\omega)|)
	\end{aligned}
	\end{equation}
	and
	\begin{equation}\nonumber
	\begin{aligned}
		&T^{\omega}_{\lambda_2}(u)(x,t)\le u(y_1)+A^{\omega}(t-\lambda_1,y_1;t,x)+\lambda_1\alpha(\omega)+
		\int^{t-\lambda_1}_{t-\lambda_2}L(y_1,0,s,\omega)ds\\
		&+ (\lambda_2-\lambda_1)\alpha(\omega)
		\le T^{\omega}_{\lambda_1}(u)(x,t)+\epsilon+(\lambda_1-\lambda_2)(|\alpha(\omega)|+|f(\omega)|)
	\end{aligned}
	\end{equation}
	Since $\epsilon>0$ is arbitrary, we have
	\[
	|T^{\omega}_{\lambda_1}u(x,t)-T^{\omega}_{\lambda_2}u(x,t)| \le (\lambda_1-\lambda_2)(|\alpha(\omega)|+|f(\omega)|)
	\]
\end{proof}

\begin{lemma}
	When $\omega$ is fixed, Lax-Oleinik operator is unifromly bounded for continuous function when $\lambda>0$.
\end{lemma}
\begin{proof}
	Step 1.	Fix $\omega \in \Omega$, $ t\in R$. Define the sequences
	$M_n(\omega)=\max_{x \in M}T^{\omega}_{n}(0)(x,t)$ and $m_n(\omega)=\min_{x \in M}T^{\omega}_{n}(0)(x,t)$ where $0$ is the zero function on $M$. From Lemma 5.1, the function $T^{\omega}_{n}(0)$,$n\ge 1$, are equi-semi-concave, there exists a constant $K(\omega)$ such that
	\begin{equation}{\nonumber}
	0 \le M_n(\omega) - m_n(\omega) \le K(\omega).
	\end{equation}
	for $n \ge 1$. We claim that $M_{n+m}(\omega) \le M_{n}(\omega) + M_{m}(\omega)$. This follows from the inequalities
	\begin{equation}{\nonumber}
	T^{\omega}_{n+m}(0)(x,t)=T^{\omega}_{m}(T^{\omega}_{n}(0)(x,t) \le T^{\omega}_{m}(M_n(\omega))(x,t)\le M_n(\omega) + T^{\omega}_{m}(0)(x,t)
	\end{equation}
	Hence by a classical result on subadditive sequences, we have $\lim \frac{M_n(\omega)}{n}=\inf\frac{ M_n(\omega)}{n}$. We denote by $-\beta(\omega)$ this limit. In the same way, the sequence $-m_{n}(\omega)$ is subadditive, hence $\frac{m_n(\omega)}{n} \to \sup \frac{m_n(\omega)}{n}$.This limit is also $-\beta(\omega)$ since $0\le M_n(\omega)-m_n(\omega) \le K$.Note that $m_1(\omega) \le -\beta(\omega) \le M_1(\omega)$, so that $\beta(\omega)$ is indeed a finite number. We have, for all $n \le 1$, 
	\begin{equation}{\nonumber}
	-K(\omega)-n\beta(\omega) \le m_n(\omega) \le -n\beta(\omega) \le M_n(\omega) \le K(\omega)-n\beta(\omega)
	\end{equation}
	Now for all $u \in C(M,R)$,$n\in N$ and $x \in M$, we have 
	\[
	\min_{M}u-K(\omega) \le \min_{M}u + m_n(\omega) + n\beta(\omega) \le T^{\omega}_{n}u(x,t) + n\beta(\omega) \le 
	\max_{M}u + M_{n}(\omega) +n \beta(\omega) \le \max_{M}u+K(\omega)
	\]
	
	Hence, for all $u \in C(M,R)$,$n \in N$,$x \in M$, we have
	\begin{equation}{\nonumber}
	\frac{\min_{M}u-K(\omega)}{n} \le \frac{T^{\omega}_{n}u(x,t)}{n} +\beta(\omega) \le \frac{\max_{M}u+K(\omega)}{n}
	\end{equation}
	Step 2.On one hand, from Proposition 4.1 , we can find a sequence of measure $\frac{1}{n_k}[\gamma_{n_k}]$ where $n_k$ is a sequence of increasing intergers towards $+\infty$, $\gamma_{n_k}$ is a closed absolutely continuous curve from $[t-n_k,t]$ to $M$, such that $\forall \epsilon>0$, we can find an interger $N$, such that, when $k \ge N$, we have
	\begin{equation}{\nonumber}
	-\alpha(\omega)\le \frac{1}{n_k}\int_{t-n_k}^{t}L(\gamma_{n_k}(s),\dot{\gamma}_{n_k}(s),s,\omega)ds \le -\alpha(\omega) + \epsilon
	\end{equation}
	Hence, we have
	\begin{equation}{\nonumber}
 T^{\omega}_{n_k}(0)(\gamma_{n_k}(t),t) \le \int_{-n_k+t}^{t}L(\gamma_{n_k}(s),\dot{\gamma}_{n_k}(s),s,\omega)dt+n_k\alpha(\omega) \le n_k \epsilon
	\end{equation}
	Therefore, we have
	\begin{equation}{\nonumber}
	 \frac{1}{n_k}T^{\omega}_{n_k}(0)(\gamma_{n_k}(t),t) \le \epsilon
	\end{equation}
On the other hand, there is	$x \in M$ such that 
\[
T^{\omega}_{n_k}(0)(\gamma_{n_k}(t),t)=A^{\omega}(x,t-n_k;\gamma_{n_k}(t),t)+n_k\alpha(\omega)
\]
We assume that $\gamma_1 \in \Sigma^{\omega}_{m}(x,t-n_k;\gamma_{n_{k}}(t),t)$.
there is an absolutely continuous path $\gamma:[t-n_k-1,t-n_k] \to M$ such that $|\dot{\gamma}(s)|\le dist(M)$, and $\gamma(t-n_k)=x$, $\gamma(t-n_k-1)=\gamma_{n_k}(t)$. Then we know that 
\[
\int^{t-n_k}_{t-n_k-1}|L(\gamma(s),\dot{\gamma}(s),s,\omega)|ds \le C(\omega)
\]
Construct a new closed curve $\tilde{\gamma}:[t-n_k-1,t] \to M$ in the following way: when $s \in [t-n_k,t]$, $\tilde{\gamma}(s)=\gamma_1(s)$; when $s \in [t-n_k-1,t-n_k]$, $\tilde{\gamma}(s)=\gamma(s)$. From proposition 4.1, we know that 
\[
\int^{t}_{t-n_k-1}L(\tilde{\gamma}(s),\dot{\tilde{\gamma}}(s),s,\omega)ds \ge -(n_k+1)\alpha(\omega)
\]
Hence, we have
\begin{equation}\nonumber
\begin{aligned}
&T^{\omega}_{t-n_k}(0)(\gamma_{n_k}(t),t)=n\alpha(\omega)+ \int^{t}_{t-n_k-1}L(\tilde{\gamma}(s),\dot{\tilde{\gamma}}(s),s,\omega)-\int^{t-n_k}_{t-n_k-1}L(\gamma(s),\dot{\gamma}(s),s,\omega)ds \\
&\ge -\alpha(\omega)-C(\omega)\\
\end{aligned}
\end{equation}
So we have 
\[
-\frac{\alpha(\omega)+C(\omega)}{n_k} \le \frac{T^{\omega}_{t-n_k}(0)(\gamma_{n_k}(t),t)}{n_k} \le \epsilon
\]	
	So combine the result with step 2, let $k$ tends to $\infty$, we know that $\beta(\omega)=0$. Hence,
	\begin{equation}{\nonumber}
	\min_{M}u-K(\omega) \le T^{\omega}_{n}u(x,t)\le \max_{M}u+K(\omega)
	\end{equation}
	
	Use Lemma 6.1, $T^{\omega}_{\lambda}u(x,t)$ is lipschitz for $\lambda$, with Lipschitz constant $|\alpha(\omega)|+|f(\omega)|$, so we have 
	\[
	\min_{M}u-K(\omega)-|\alpha(\omega)|-|f(\omega)| \le T^{\omega}_{\lambda}u(x,t) \le \max_{M}u + K(\omega) + |\alpha(\omega)| + |f(\omega)|
	\]

\end{proof}

\begin{lemma}If the variables $t,s \in R$, $x,y \in M$, $n \in \mathbb{N}$, $u \in C_{0}(M,R)$ are fixed, $A^{\omega}(s,y;t,x)$, $T^{\omega}_{\lambda}(u)(x,t)$, $T^{\omega}_{\lambda,+}(u)(x,t)$ and $u^{\omega}(x,t)$, $u^{\omega}_{+}(x,t)$ are random variables over the probability space $\Omega$.
\end{lemma}

\begin{proof}
	The extreme curves of the Action are $C^2$ since they satisfy the lagrangian equations.
	
	Let $C^{0}([s,t],TM)$ denote the space of continuous function from $[s,t]$ to $TM$ with uniform norm in  $TM$. $M$ is a manifold, by whitney' embedding theorem, there exists $m \in N$, such that $M$ can be embedded isomorphically into $R^m$. Therefore, $TM$ can be embedded as a submanifold of $R^{2m}$. By Stone Weierstrass theorem, the space of continuous functions from $[s,t]$ into $R$ is separable.  Therefore, the space of continuous functions from $[s,t]$ to $R^{2m}$ is separable, and as its subset, $C^{0}([s,t],TM)$ is also separable. Since the extreme curves of the Action is a subset of $C^{0}([s,t],TM)$,  So we can pick up a countable dense subset $\{\gamma_i\}_{i \in N}$,such that
	\begin{equation}\nonumber
	A^{\omega}(s,y;t,x)=\min_{i \in N}\int^{t}_{s}L(\gamma_i(\sigma),\dot{\gamma_i}(\sigma),\sigma,\omega)d\sigma
	\end{equation}
	By Fubini Thoerem, $\int^{t}_{s}L(\gamma_i(\sigma),\dot{\gamma_i}(\sigma),\sigma,\omega)d\sigma$ is measurable function on $\Omega$. As a minimun of countable measurable functions, $A^{\omega}(s,y;t,x)$ is measurable.
	
	$M$ is a separable metric space, $u(y)$, $A^{\omega}(t-\lambda,y;t,x)$ are continuous functions with $y$, so we can find a countable dense set $\{y_i\}_{i\in N}$ in $M$ such that 
	\begin{equation}\nonumber
	T^{\omega}_{\lambda}(u)(x,t)=\min_{i \in N}\{u(y_i)+ A^{\omega}(t-\lambda,y_i;t,x)+\lambda\alpha(\omega)\}
	\end{equation}
	As a minimum of countable measurable functions, $T^{\omega}_{\lambda}(u)(x,t)$ is a random variable.
	
	Since $T^{\omega}_{\lambda}(u)(x,t)$ is uniformly continuous with $\lambda$, so we can a pick a sequnce $\lambda_n=\sum_{k=1}^{n}\frac{1}{k} \to \infty$.
	Then, we have
	\begin{equation}\nonumber
	u^{\omega}(x,t)=\liminf_{n \to \infty}T^{\omega}_{\lambda_n}(u)(x,t)
	\end{equation}
	So as the infimum limit of a countable sequence of measurable function, $u^{\omega}(x,t)$ is measurable over $\Omega$.
	
	The measurability of $T^{\omega}_{\lambda,+}(u)(x,t)$,$u^{\omega}_{+}(x,t)$ can be proved in a similar way.
\end{proof}

\begin{lemma}
	Fix $u\in C_{0}(M,R)$, for $\lambda>0 ,t,s \in R$, $x \in M $, we have the following formula:
	1.\begin{equation}\nonumber
	T^{\theta(s)\omega}_{\lambda}(u)(x,t)=T^{\omega}_{\lambda}(u)(x,t+s)
	\end{equation}
	2.\begin{equation}\nonumber
	u^{\omega}(x,t)=T^{\omega}_{t-s}(u^{\omega}(\;,s))(x,t)=\min_{y\in M}\{u^{\omega}(y,s)+A^{\omega}(s,y;t,x)+(t-s)\alpha(\omega)\}
	\end{equation}
	3.\begin{equation}\nonumber
	u^{\theta(s)\omega}(x,t)=u^{\omega}(x,t+s)
	\end{equation}
	4.\begin{equation}\nonumber
	u^{\omega}(x,t)=u^{\omega}(x,t+1)
	\end{equation}
	These formula have similar versions for $T^{\omega}_{\lambda,+}(u)(x,t)$,$u^{\omega}_{+}(x,t)$.
\end{lemma}

\begin{proof}
	First of all, the formula 1 can be derived directly from $L(x,v,t,\theta(s)\omega)=L(x,v,t+s,\omega)$.
	Secondly,by definition, $\forall \epsilon>0$,there exists $y\in M$,such that 
	\begin{equation}\nonumber
	T^{\omega}_{t-s}u^{\omega}( ,s)(x,t)+\epsilon=u^{\omega}(y,s)+ A^{\omega}(s,y;t,x)+(t-s)\alpha(\omega)
	\end{equation}
	By defnition, there exists a sequence $\lambda_i \to +\infty$ as $i \to +\infty$, such that 
	\begin{align*}
		u^{\omega}(y,s)=\liminf_{i \to +\infty}T^{\omega}_{\lambda_i}(u)(y,s)
	\end{align*}
	Therefore, we have 
	\begin{equation}\nonumber
	\begin{aligned}
	&T^{\omega}_{t-s}u^{\omega}( ,s)(x,t)+\epsilon
	=\liminf_{i \to +\infty}T^{\omega}_{\lambda_i}(u)(y,s)+A^{\omega}(s,y;t,x)+(t-s)\alpha(\omega) \\
	&\ge \liminf_{i\to +\infty}T^{\omega}_{\lambda_i+t-s}(u)(x,t) \ge \liminf_{\lambda \to +\infty}T^{\omega}_{\lambda}(u)(x,t)=u^{\omega}(x,t)
	\end{aligned}
	\end{equation}
	Since $\epsilon>0$ is arbitrary, we can get that
	\begin{equation}\nonumber
	T^{\omega}_{t-s}u^{\omega}( ,s)(x,t)\ge u^{\omega}(x,t)
	\end{equation}
	Conversely, there exists $\lambda_k \to +\infty$ as $k \to +\infty$, such that 
	\begin{equation}\nonumber
	u^{\omega}(x,t)=\lim_{k \to +\infty}T^{\omega}_{\lambda_k}(u)(x,t)=\lim_{k \to +\infty}T^{\omega}_{t-s}(T^{\omega}_{\lambda_k-t+s}(u))(x,t)
	\end{equation}
	By definition, $\forall \epsilon > 0$, there exists $\{q_k\}_{k\in N} \in M$,such that
	\begin{equation}\nonumber
	T^{\omega}_{\lambda_k}(u)(x,t)\ge T^{\omega}_{\lambda_k-t+s}(u)(q_k,s)+A^{\omega}(q_k,s;x,t)+(t-s) \alpha(\omega)-\epsilon
	\end{equation}
	Since $M$ is a compact manifold, without loss of generality, we can assume that $q_k \to q$ as $k \to \infty$ for some $q\in M$. so 
	\begin{equation}\nonumber
	\begin{aligned}
		&u^{\omega}(x,t)\ge \lim_{k\to \infty}T^{\omega}_{\lambda_k-t+s}(u)(q_k,s)+A^{\omega}(q_k,s;x,t)+(t-s) \alpha(\omega)-\epsilon  \\
		&\ge \liminf_{k \to \infty}T^{\omega}_{\lambda_k-t+s}(u)(q,s)+ A^{\omega}(q,s;x,t)+ (t-s)\alpha(\omega)-\epsilon\\
		&\ge u^{\omega}(q,s)+A^{\omega}(q,s;x,t)+(t-s)\alpha(\omega)-\epsilon\\
		&\ge T^{\omega}_{t-s}(u( ,s))(x,t)-\epsilon
	\end{aligned}
	\end{equation}
	Since $\epsilon>0$ is arbitrary, we have that
	\begin{equation}\nonumber
	u^{\omega}(x,t) \ge T^{\omega}_{t-s}u^{\omega}( ,s)(x,t)
	\end{equation}
	Consequently, the formula 2 holds.
	
	Thirdly, the formula 3 arises directly from formula 1.
	
	Finally, we can get the formula 4 from $L(x,v,t+1,\omega)=L(x,v,t,\omega)$.
\end{proof}

\begin{theorem}
	Fix $u \in C_{0}(M,R)$, $u^{\omega}(x,t)$ is a viscosity solution of the Hamilton-Jacobi Equation
	\begin{equation}\nonumber
	\partial_t u(x,t) + H(x,\partial_x u(x,t),t,\omega)=\alpha(\omega) 
	\end{equation}
\end{theorem}

\begin{proof}
	Fix $\omega \in \Omega$, for any $(x_0,t_0) \in M$. We can find a local chart $(U,\phi)$, $U \in M$,$\phi:U \to \phi(U) \in R^n$ is a diffeomorphism, $(x_0,t_0) \in U$. Without loss of generality, we assume that $U \in R^n$.
	
	Step 1.To show $u^{\omega}(x,t)$ is a subsolution of the Hamilton-Jacobi equation, we need to prove that if $(p_x,p_t) \in D^{+}u^{\omega}(x_0,t_0)$, 
	$p_t+H(x,p_x,t,\omega) \le \alpha(\omega)$.
	
	$\forall v \in T_{x_0}M$, since $(p_x,p_t) \in D^{+}u^{\omega}(x_0,t_0)$, we have
	\begin{equation}{\nonumber}
	\limsup_{h \to 0+} \frac{u^{\omega}(x_0-hv,t_0-h)-u^{\omega}(x_0,t_0)+h(p_t+p_x \cdot v)}{h\sqrt{1+{|v|}^2}}\le 0
	\end{equation}
	Which is equivalent to 
	\begin{equation}{\nonumber}
	\limsup_{h \to 0+} \frac{u^{\omega}(x_0-hv,t_0-h)-u^{\omega}(x_0,t_0)}{h}\le -p_t - p_x \cdot v
	\end{equation}
	Since $U$ is open, there exists $\sigma>0$, such that$\{\gamma(t)=x-s\cdot v| 0 \le s \le \sigma \}\in U$, from lemma 4, we know that when $0<h\le\sigma$, we have
	\begin{equation}{\nonumber}
	u^{\omega}(x_0,t_0) \le u^{\omega}(x_0-h\cdot v,t_0-h)+\int^{t_0}_{t_0-h}L(\gamma(s),\dot{\gamma}(s),s,,\omega)ds + h\cdot \alpha(\omega)
	\end{equation}
	Then,
	\begin{equation}{\nonumber}
	-\alpha(\omega)-\liminf_{h \to 0+}\frac{1}{h}\int^{t_0}_{t_0-h}L(\gamma(s),\dot{\gamma}(s),s,\omega)ds \le \limsup_{h \to 0+}\frac{u(x_0-hv,t_0-h)-u(x_0,t_0)}{h} \le  -p_t - p_x \cdot v
	\end{equation}
	Hence, we have
	\begin{equation}{\nonumber}
	-\alpha(\omega)-L(x_0,v,t_0,\omega) + p_x \cdot v+p_t\le 0
	\end{equation}
	Therefore, 
	\begin{equation}{\nonumber}
	p_t+H(x_0,v,t_0,\omega)-\alpha(\omega) =p_t+ \sup_{v \in TM}\{p_x \cdot v -L(x_0,v,t_0,\omega)\}-\alpha(\omega) \le 0
	\end{equation}
	
	Step 2.To show $u^{\omega}(x,t)$ is a subsolution of the Hamilton-Jacobi equation, we need to prove that if $(p_x,p_t) \in D^{+}u^{\omega}(x_0,t_0)$, $p_t + H(x,p_x,t,\omega) \ge \alpha(\omega)$.

	From Lemma 6.4, we know that there exists $y \in M$, an absolute curve $\{\gamma(t)| t_0-1 \le t \le t_0,\gamma(t_0-1)=y,\gamma(t_0)=x_0\}$,such that 
	\begin{equation}{\nonumber}
	u^{\omega}(x_0,t_0)=u^{\omega}(y,t_0-1)+\int^{t_0}_{t_0-1}L(\gamma(t),\dot{\gamma}(t),t,\omega)dt+(t_0-t)\alpha(\omega)
	\end{equation}
	Since $U$ is open, there exists $\sigma>0$, such that $\{\gamma(t)|t_0-\sigma \le t \le t_0 \}\in U$.
	Then we have 
	\begin{equation}{\nonumber}
	u^{\omega}(x_0,t_0)=u^{\omega}(\gamma(t_0-h),t_0-h)+\int^{t_0}_{t_0-h}L(\gamma(t),\dot{\gamma}(t),t,\omega)dt + h\alpha(\omega)
	\end{equation}

	Let $w=\dot{\gamma}(t_0)$,since $(p_t,p_x) \in D^{-}u^{\omega}(x_0,t_0)$, we know that
	\begin{equation}{\nonumber}
	\liminf_{h \to 0+}\frac{u^{\omega}(\gamma(t_0-h),t_0-h)-u^{\omega}(\gamma(t_0),t_0)}{h}\ge -p_x\cdot w -p_t
	\end{equation} 
	Hence, we have
	\begin{equation}{\nonumber}
	-p_x\cdot w-p_t \le -\limsup_{h \to 0+} \frac{1}{h}\int^{t_0}_{t_0-h}L(\gamma(t),\dot{\gamma}(t),t,\omega)dt -\alpha(\omega) \le -L(x_0,w,t_0,\omega) -\alpha(\omega)
	\end{equation}
	Therefore,we have 
	\begin{equation}{\nonumber}
	H(x,p_x,t,\omega)=\sup_{v \in T_{x_0}M}\{p_x\cdot v-L(x_0,v,t_0,\omega)\}\ge p_x\cdot w-L(x_0,w,t_0,\omega) \ge \alpha(\omega)-p_t
	\end{equation}
	From Step 1 and Step 2, we know that $u^{\omega}(x,t)$ is a viscosity solution of Hamilton-Jacobi equation.
\end{proof}

\section{Global Minimizer}
In this section, we discuss the global minimizer and invariant measure under the skew-product dynamics system.
\begin{definition}
	When $\omega$ is fixed, 
	
	1.an absolutely continuous orbit $\{\gamma^{\omega}(t), t\in R \}$ is called a global minimizer of the Lagrangian if for any fixed time interval $[s,t]$, we have
	\begin{equation}\nonumber
	A^{\omega}(s,\gamma^{\omega}(s);t,\gamma^{\omega}(t))=\int_{s}^{t}L(\gamma^{\omega}(\sigma),\dot{\gamma^{\omega}}(\sigma),\sigma,\omega)d\sigma
	\end{equation}
	
	2. fix $t_0 \in R$,an absolutely continuous orbit $\{\gamma^{\omega}(t), t \le t_0 \}$ is called a calibrated curve of $u^{\omega}(x,t)$,if for any $s<t \le t_0$, we have
	\begin{equation}\nonumber u^{\omega}(\gamma^{\omega}(t),t)=u^{\omega}(\gamma^{\omega}(s),s)+ A^{\omega}(s,\gamma^{\omega}(s);t,\gamma^{\omega}(t))+(t-s)\alpha(\omega)
	\end{equation}
	
	3. fix $t_0 \in R$,an absolutely continuous orbit $\{\gamma^{\omega}(t), t \ge t_0 \}$ is called a calibrated curve of $u^{\omega}_{+}(x,t)$,if for any $s > t \ge t_0$, we have
	\begin{equation}\nonumber u^{\omega}_{+}(\gamma^{\omega}(t),t)=u^{\omega}_{+}(\gamma^{\omega}(s),s)+ A^{\omega}(t,\gamma^{\omega}(t);s,\gamma^{\omega}(s))+(s-t)\alpha(\omega)
	\end{equation}
\end{definition}

\begin{lemma}
	We fix $\omega$ and $t_0$, if $u^{\omega}(x,t_0)$ is differentiable at $x_0 \in M$, there is a unique calibrated curve with the end point $(x_0,t_0)$, for $u^{\omega}(x,t)$, $t \le t_0$; denoted as  $\gamma^{\omega,-}_{x,t_0}$. Similarly,if $u^{\omega}_{+}(x,t_0)$ is differentiable at $x_0 \in M$, there is a unique calibrated curve with the end point $(x_0,t_0)$, for $u^{\omega}_{+}(x,t)$, $t \ge t_0$; denoted as  $\gamma^{\omega,+}_{x,t_0}$.
\end{lemma}
\begin{proof}
	By lemma 3.3, for any $s<t_0$, $x_0 \in B^{\omega}_{t_0}$, we have 
	\begin{equation}\nonumber
	u^{\omega}(x_0,t)=\min_{y\in M}\{u^{\omega}(y,s)+A^{\omega}(s,y;t,x)+(t-s)\alpha(\omega)\}
	\end{equation}
	Since $u^{\omega}(y,s)$ and $A^{\omega}(s,y;t,x)$ is semiconcave with respect to $y$, so there is $x(s) \in M$ such that 
	\begin{equation}\nonumber
	u^{\omega}(x_0,t)=u^{\omega}(x(s),s)+A^{\omega}(s,x(s);t,x_0)+(t-s)\alpha(\omega)
	\end{equation}
	and $\partial_x u^{\omega}(x(s),s)+\partial_y A^{\omega}(s,x(s);t,x_0)=0$. By Corollary 5.1 , $ \Sigma^{\omega}_{m}(s,x(s);t_0,x_0)$ contains only one lagrangian trajectory $(x(t),p(t))$, $p(t)=\partial_v L(x(s),\dot{x}(s),s,\omega)$ with $s \le t \le t_0$, with $x(t_0)=x_0$, $p(s)=\partial_x u^{\omega}(x(s),s)=-\partial_y A^{\omega}(s,x(s);t,x_0)$ where $p(t)=\partial_x u^{\omega}(x(t_0),t_0)=\partial_x A^{\omega}(s,x(s);t,x_0)$. Therefore, it is obvious that $x(s)$ is unique. 
	
	For any $s_1 < s_2 < t_0$, the above lagrangian trajectory coincides for $s_2 \le t \le t_0$. So we can extend the trajectory $x(s)$ for time $-\infty < s \le t_0$, with $x(t_0)=x_0$. This is a unique calibrated curve for $u^{\omega}(x,t)$, $t \le t_0$ with the end point a $x(t_0)=x_0$.
\end{proof}

\begin{lemma}
	When $\omega$, $t_0$ is fixed, the set $B^{\omega}_{t_0}$ is defined as the subsets of $M$ where the function $u^{\omega}(y,t_0)+u^{\omega}_{+}(y,t_0)$ attains the minimum. We have the following results:
	
	1. $B^{\omega}_{t_0}$ is a closed nonempty subset of $M$. $ u^{\omega}(x,t)$, $ u^{\omega}_{+}(x,t)$ are differentiable over $B^{\omega}_{t_0}$;
	
	2. for any $x_0 \in B^{\omega}_{t_0}$, $\partial_{x}u^{\omega}(x_0,t)+\partial_{x}u^{\omega}_{+}(x_0,t)=0$. 
	
	3. $B^{\theta(s)\omega}_{t}=B^{\omega}_{t+s}$, $B^{\omega}_{t+1}=B^{\omega}_{t}$
\end{lemma}
\begin{proof}
	It is obvious that $B^{\omega}_{t_0}$ is closed nonempty subset. Since $u^{\omega}(x,t_0)$,$u^{\omega}_{+}(x,t_0)$ are both semiconcave function over $M$, they are differentiable over $B^{\omega}_{t_0}$ and  $\partial_{x}u^{\omega}(x_0,t_0)+\partial_{x}u^{\omega}_{+}(x_0,t_0)=0$. By lemma 3.3,   $B^{\theta(s)\omega}_{t}=B^{\omega}_{t+s}$, $B^{\omega}_{t+1}=B^{\omega}_{t}$ are obvious.
\end{proof}

\begin{proposition}
	When $\omega$ is fixed, for any $t_0 \in R$, if the following hypothesis holds:
	
	$u^{\omega}(\gamma^{\omega}(t_0),t_0)+u^{\omega}_{+}(\gamma^{\omega}(t_0),t_0)=\min_{x \in M}\{u^{\omega}(x,t_0)+u^{\omega}_{+}(x,t_0) \}$
	
	$\{\gamma^{\omega}(t),t \le t_0\}$ is a calibrated curve of $u^{\omega}(x,t)$,
	
	$\{\gamma^{\omega}(t),t \ge t_0\}$ is a calibrated curve of $u^{\omega}_{+}(x,t)$,
	
	then $\{\gamma^{\omega}(t), t\in R \}$ is a global minimizer of the lagrangian.
\end{proposition}
\begin{proof}
	If the conclusion falses, there exists time $t_1<t_2$, and an absolutely continuous curve $\{\gamma_1(t),t_1\le t \le t_2 \}$ with the terminal points $\gamma_1(t_1)=\gamma^{\omega}(t_1)$,$\gamma_1(t_2)=\gamma^{\omega}(t_2)$, such that 
	\begin{equation}\nonumber
	\int_{t_1}^{t_2}L(\gamma_1(\sigma),\dot{\gamma_1}(\sigma),\sigma,\omega)d\sigma <\int_{t_1}^{t_2}L(\gamma^{\omega}(\sigma),\dot{\gamma^{\omega}}(\sigma),\sigma,\omega)d\sigma
	\end{equation}
	then there exist $t_0 \in (t_1,t_2)$ with $\gamma_1(t_0)\ne\gamma^{\omega}(t_0)$.We have the following:
	\begin{equation}\nonumber
	\begin{aligned}
		&u^{\omega}(\gamma_1(t_0),t_0)+u^{\omega}_{+}(\gamma_1(t_0),t_0) 
		\le u^{\omega}(\gamma_1(t_1),t_1)+\int_{t_1}^{t_0}L(\gamma_1(\sigma),\dot{\gamma_1}(\sigma),\sigma,\omega)d\sigma \\
		& + (t_0-t_1)\alpha(\omega)+u^{\omega}_{+}(\gamma_1(t_2),t_2)+ \int_{t_0}^{t_2}L(\gamma_1(\sigma),\dot{\gamma_1}(\sigma),\sigma,\omega)d\omega + (t_2-t_0)\alpha(\omega) \\
		&=u^{\omega}(\gamma_1(t_1),t_1)+ u^{\omega}_{+}(\gamma_1(t_2),t_2)+\int_{t_1}^{t_2}L(\gamma_1(\sigma),\dot{\gamma_1}(\sigma),\sigma,\omega)d\omega+ (t_2-t_1)\alpha(\omega)
	\end{aligned}
	\end{equation}
	and
	\begin{equation}\nonumber
	\begin{aligned}
		&u^{\omega}(\gamma^{\omega}(t_0),t_0)+u^{\omega}_{+}(\gamma^{\omega}(t_0),t_0)=u^{\omega}(\gamma^{\omega}(t_1),t_1)+ \int_{t_1}^{t_0}L(\gamma^{\omega}(\sigma),\dot{\gamma^{\omega}}(\sigma),\sigma,\omega)d\omega\\
		&+ (t_0-t_1)\alpha(\omega)+u^{\omega}_{+}(\gamma^{\omega}(t_2),t_2) + \int_{t_0}^{t_2}L(\gamma^{\omega}(\sigma),\dot{\gamma^{\omega}}(\sigma),\sigma,\omega)d\sigma + (t_2-t_0)\alpha(\omega)\\
		&=u^{\omega}(\gamma^{\omega}(t_1),t_1)+u^{\omega}_{+}(\gamma^{\omega}(t_2),t_2)+\int_{t_1}^{t_2}L(\gamma^{\omega}(\sigma),\dot{\gamma^{\omega}}(\sigma),\sigma,\omega)d\sigma+(t_2-t_1)\alpha(\omega)
	\end{aligned}
	\end{equation}
	So, $u^{\omega}(\gamma_1(t_0),t_0)+u^{\omega}_{+}(\gamma_1(t_0),t_0)<u^{\omega}(\gamma^{\omega}(t_0),t_0)+u^{\omega}_{+}(\gamma^{\omega}(t_0),t_0)$, this contradicts the hypothesis. Hence we can complete the proof.
\end{proof}

\begin{corollary}
	For fixed $\omega$, $t_0$; for any $x_0 \in B^{\omega}_{t_0}$,  a global minimizer $\gamma^{\omega}_{x_0,t_0}$ exists, $\gamma^{\omega}_{x_0,t_0}(t_0)=x_0$.In addition, the set of global minimizer is nonempty for each $\omega \in \Omega$.
\end{corollary}


\begin{assumption}
$(\Omega,\mathscr{F},\mathbb{P})$ is a Polish Space.
\end{assumption}
 
Starting from here, we assume Assumption 7.1 holds.

From the analysis above, fix $\omega$, the set of global minimizer is nonempty. We denote the set of global minimizer as
 $H_{\omega}=\{\gamma^{\omega}:R \to M | \gamma^{\omega}
  $is a global minimizer of action $ A^{\omega}\}$.
  
   If $\gamma^{\omega}$ is a global minimizer, $\gamma^{\omega}$ satisfies the Euler-Lagrange equation. $\gamma^{\omega}$ is decided by $(\gamma^{\omega}(0),{\dot{\gamma}}^{\omega}(0))$. So we can define the set $G_{\omega}=\{(\gamma^{\omega}(0),{\dot{\gamma}}^{\omega}(0)|\gamma^{\omega}$ is a global minimizer of the action $A^{\omega} \}$.  $G_{\omega}$ is a compact subset of $TM$, due to Tonelli Theorem and continuity of Lagrangian Action.

 $G_{\omega}$ and $H_{\omega}$ corresponds one to one, $L(x,v,t+s,\omega)=L(x,v,t,\theta(s)\omega)$. if $\gamma^{\omega}(t)$ is a global minimizer of the action $A^{\omega}$, then $\gamma^{\theta(s)\omega}(t)=\gamma^{\omega}(t+s)$ is the global minimizer of $A^{\theta(s)\omega}$, we define the one-to-one map $\Theta(s): G_{\omega} \to G_{\theta(s)\omega}$ as follows, if $(\gamma^{\omega}(0),\dot{\gamma}^{\omega}(0)) \in G_{\omega}$, let $ \Theta(s)(\gamma^{\omega}(0),{\dot{\gamma}}^{\omega}(0))=(\gamma^{\theta(s)\omega}(0),{\dot\gamma}^{\theta(s)\omega}(0))=(\gamma^{\omega}(-s),{\dot{\gamma}}^{\omega}(-s))$.When $\omega$ is fixed, $\Theta(-s)$ is the Lagrangian flow the global minimizer on $TM$. $\Theta(s+t)=\Theta(s)\Theta(t)$.$G_{\omega}=G_{\theta(n)\omega}$.

Let $G=\bigcup_{\omega \in \Omega}G_{\omega}\times \{\omega\} \in TM \times \Omega$. Define $\Gamma(s):G \to G$, if $(x,v,\omega) \in G$, $\Gamma(s)(x,v,\omega)=({\Theta(s)|}_{G_{\omega}}(x,v),\theta(s)\omega)$.It is clear that $\Gamma(s+t)=\Gamma(s)\Gamma(t)$. We define the set $\Lambda=\{\mu|\mu$ is a probability measure with support in $G \} $. $\Lambda$ is a convex set. 
\begin{lemma}
Fix $\omega \in \Omega$, $ P(G_{\omega})$ is the convex hull spanned by $\{\delta_{(x,v)}|(x,v) \in G_{\omega}\}$.
\end{lemma}
\begin{proof}
We claim that $F_{\omega}=\{\sum^{n}_{k=1}a_k \delta_{(x_k,v_k)}|n \in N, \sum^{n}_{k=1}a_k=1, 0 \le a_k \le 1,(x_k,v_k) \in G_{\omega} \}$ are dense in $P(G_{\omega})$. $G_{\omega}$ is a compact subset of the complete separable metric space $TM$ with distance $d_{TM}$. The proof that $F_{\omega}$ is dense in $P(G_{\omega})$ is as same as that of Lemma 3.3. See \cite{billingsley2013convergence}. Hence, $P(G_{\omega})$ is the convex hull spanned by $\{\delta_{(x,v)}|(x,v) \in G_{\omega}\}$
\end{proof}

\begin{lemma}
If $G$ is measurable, For any $\mu \in \Lambda$, for almost every $\omega \in \Omega$, there exists $\mu_{\omega} \in P(G_{\omega})$, such that for any measurable function $f:TM \times \Omega \to R$,we have 
\[
\int_{G}f(x,v,\omega)d\mu= \int_{\Omega}\int_{G_{\omega}}f(x,v,\omega)d\mu_{\omega}d\omega
\]
\end{lemma}
\begin{proof}
If $\mu \in \Lambda$, then $\mu \in P(TM \times \Omega)$, since $TM$ and $\Omega$ are polish spaces, by Theorem 3.1, for almost every $\omega \in \Omega$, there exists $\nu_{\omega} \in P(TM)$, such that for any measurable function $f:TM \times \Omega \to R$, we have
\[
\int_{TM \times \Omega}f(x,v,\omega)d\mu=\int_{\Omega}\int_{TM}f(x,v,\omega)d\nu_{\omega}d\omega
\]
Then, let
$\mu_{\omega}=\chi_{G_{\omega}}\nu_{\omega} \in P(G_{\omega})$, we know that
\begin{equation}\nonumber
\begin{aligned}
&\int_{TM \times \Omega}f(x,v,\omega)d\mu=\int_{TM \times \Omega}f(x,v,\omega)\chi_{G} d\mu\\
&=\int_{\Omega}\int_{TM}f(x,v,\omega)\chi_{G_{\omega}}d\nu_{\omega}d\omega=\int_{\Omega}\int_{TM}f(x,v,\omega)d\mu_{\omega}d\omega\\
\end{aligned}
\end{equation}

\end{proof}

Le $\pi_{\Omega}:TM \times \Omega \to \Omega$ as the canonical projection.
\begin{lemma}
If $\mu \in	\Lambda$, and $\mu$ is invariant under the transformation of $\Gamma(s)$, $s \in R$. We decompose $d\mu=d\mu_{\omega}d\omega$, where $\mu_{\omega}$ is a probability measure on $G_{\omega}$ for almost all $\omega \in \Omega$. We have $d\mu_{\omega}d\omega=d\Theta(s)^{*}\mu_{\theta(-s)\omega} d\theta(-s)\omega$. If $\pi_{\Omega}\mu$ is invariant under the transformation $\{\theta(s),s \in R\}$, we have $\mu_{\omega}=\Theta(s)^{*}\mu_{\theta(-s)\omega}$ for almost every $\omega \in \Omega$.
\end{lemma}
\begin{proof}
We assume that $\mu$ is invariant under $\{\Gamma(s),s \in R\}$.For any measurable function $f:TM \times T \times \Omega \to R$
\begin{equation}\nonumber
\begin{aligned}
&\int_{\Omega}\int_{G_{\omega}}f(x,v,\omega)d\mu_{\omega}d\omega=\int_{G}f(x,v,\omega)d\mu=\int_{G}f(x,v,\omega)d\Gamma^{*}(s)\mu\\
&=\int_{G}f(\Theta(s)|_{G_{\omega}}(x,v),\theta(s)\omega)d\mu=\int_{\Omega}\int_{G_{\omega}}f(\Theta(s)|_{G_{\omega}}(x,v),\theta(s)\omega)d\mu_{\omega}d\omega\\
&=\int_{\Omega}\int_{G_{\theta(-s)\omega}}f(\Theta(s)|_{G_{\theta(-s)(\omega)}}(x,v),\omega)d\mu_{\theta(-s)\omega} d\theta(-s)\omega\\
&=\int_{\Omega}\int_{G_{\theta(-s)\omega}}f(x,v,\omega)d\Theta(s)^{*}\mu_{\theta(-s)\omega} d\theta(-s)\omega\\
&=\int_{G}f(x,v,\omega)d\Theta(s)^{*}\mu_{\theta(-s)\omega} d\theta(-s)\omega\\
\end{aligned}
\end{equation}
Since $f$ is arbitrary, we have $d\mu_{\omega}d\omega=d\Theta(s)^{*}\mu_{\theta(-s)\omega} d\theta(-s)\omega$. If $\pi_{\Omega}\mu$ is invariant under $\{\theta(s),s\in R\}$, then we have $\mu_{\omega}=\Theta(s)^{*}\mu_{\theta(-s)\omega}$ for almost every $\omega \in \Omega$.
\end{proof}
\begin{theorem}
When $G$ is a measurable set in $TM \times \Omega$. If $\mu$ is an ergodic invariant measure of $\{\Theta(s),s \in R\}$,$\pi_{\Omega}(\mu)$ is invariant under the transformation $\{\theta(s),s \in R\}$. Then we know that for almost $\omega \in \Omega$, there exists $(x_{\omega},v_{\omega})\in TM$, such that $\mu_{\omega}=\delta_{(x_{\omega},v_{\omega})}$. And when $s \in R$, we have $\Theta(s)(x_{\omega},v_{\omega})=(x_{\theta(s)\omega},v_{\theta(s)\omega})$ for almost all $\omega \in \Omega$.
\end{theorem}
\begin{proof}
We assume that $\mu$ is an ergodic invariant measure under the transformation $\{\Theta(s),s \in R\}$. Since the set of invariant measure is a closed convex set. It is well known  that the ergodic measure is the extreme point of the convex set, see \cite{katok1997introduction}. By Lemma 7.4, for almost all $\omega \in \Omega$, we know that  $\mu_{\omega}$ is an extreme point of $P(G_{\omega})$, there exists $(x_{\omega},v_{\omega}) \in TM$, such that $\mu_{\omega}=\delta_{(x_{\omega},v_{\omega})}$. By Lemma 7.6, when $s \in R$, we know that $(x_{\theta(s)\omega},v_{\theta(s)\omega})=\Theta(s)(x_{\omega},v_{\omega})$ for almost every $\omega \in \Omega$.
\end{proof}

\bibliography{bib}

\end{document}